\documentclass[11pt]{amsart}
\usepackage[british]{babel}

\usepackage{amssymb,latexsym}
\usepackage{comment}
\usepackage{titletoc}
\usepackage{amsfonts, amsmath, amsthm, amssymb, fancybox, graphicx, color, times, babel,mathrsfs}
\usepackage[utf8]{inputenc}
\usepackage{multirow}
\usepackage{enumitem}
\usepackage[numbers,sort&compress]{natbib}

\usepackage[usenames,dvipsnames]{xcolor}

\usepackage[all,cmtip]{xy}
\usepackage[normalem]{ulem}
\usepackage{cancel}

\usepackage{tikz-cd}

\usepackage{a4wide}

\newtheorem {theorem}{Theorem}
\newtheorem {corollary}[theorem]{Corollary}

\newtheorem {lemma}[theorem]{Lemma}

\newtheorem {proposition}[theorem]{Proposition}

\theoremstyle{definition}
\newtheorem {remark}[theorem]{Remark}
\newtheorem {definition}[theorem]{Definition}

\numberwithin{theorem}{section}

\DeclareMathOperator{\disc}{disc}

\newcommand{\Q}{\mathbb{Q}}

\newcommand{\Z}{\mathbb{Z}}

\renewcommand{\geq}{\geqslant}
\renewcommand{\leq}{\leqslant}

\makeatletter
\newcommand{\customlabel}[2]{
   \protected@write \@auxout {}{\string \newlabel {#1}{{#2}{\thepage}{#2}{#1}{}} }
   \hypertarget{#1}{}
}
\makeatother

\usepackage{xr-hyper}
\usepackage{hyperref}

\date{}

\begin{document}
\author[I. Del Corso]{Ilaria Del Corso$^1$}\thanks{$^{[1]}$ The authors have been partially supported by MIUR (Italy) through   PRIN 2017  ``Geometric, algebraic and
    analytic methods in arithmetic'', and by Universit\`a di Pisa through  PRA 2018-19 ``Spazi di moduli, rappresentazioni e strutture combinatorie''.   
     } 
\address[]{ \parbox{\linewidth}{Ilaria del Corso \\ Dipartimento di Matematica, Università di Pisa, Largo Bruno Pontecorvo 5, 56127 Pisa, Italy \vspace{0.15cm}}}
\email{ilaria.delcorso@unipi.it}

\author[F. Ferri]{Fabio Ferri$^2$}
\thanks{$^{[2]}$The author gratefully acknowledges the funding of his PhD studentship by the University of Exeter.}
\address[]{\parbox{\linewidth}{Fabio Ferri \newline Department of Mathematics, University of Exeter, Rennes Dr, Exeter EX4 4RN, United Kingdom\vspace{0.15cm}}}
\email{ff263@exeter.ac.uk}

\author[D. Lombardo]{Davide Lombardo$^1$}
\address[]{\parbox{\linewidth}{Davide Lombardo \newline Dipartimento di Matematica, Università di Pisa, Largo Bruno Pontecorvo 5, 56127 Pisa, Italy\vspace{0.15cm}}}
\email{davide.lombardo@unipi.it}

\title{How far is an extension of $p$-adic fields from having a normal integral basis?}

\begin{abstract}
Let $L/K$ be a finite Galois extension of $p$-adic fields with group $G$. It is well-known that $\mathcal{O}_L$ contains a free $\mathcal{O}_K[G]$-submodule of finite index. We study the \textit{minimal} index of such a free submodule, and determine it exactly in several cases, including for any cyclic extension of degree $p$ of $p$-adic fields.
\end{abstract}

\maketitle

\section{Introduction}
Let $L/K$ be a Galois extension of $p$-adic fields with Galois group $G$, ramification index $e_{L/K}$ and inertia degree $f_{L/K}$, and let $\mathcal{O}_K$ (respectively $\mathcal{O}_L$) denote the ring of integers of $K$ (resp.~$L$). 
For any $p$-adic field $F$, let $e_F$ and $f_F$ be the absolute ramification index and inertia degree, $\pi_F$ a uniformiser, and $v_F$ the valuation, normalised by $v_F(\pi_F)=1$.
We will write $K[G]$ for the $K$-algebra whose elements are formal linear combinations of elements of $G$ with $K$-coefficients. We similarly define the group ring $\mathcal{O}_K[G]$ and observe that $L$ is in a natural way a $K[G]$-module (resp.~$\mathcal{O}_L$ is a $\mathcal{O}_K[G]$-module).
The classical normal basis theorem shows that in this setting $L$ is a free $K[G]$-module of rank 1. It is then natural to ask whether $\mathcal{O}_L$ is also a free $\mathcal{O}_K[G]$-module of rank 1. The answer to this question is given by the following result.
\begin{theorem}[{\cite{MR1581331}, \cite[Theorem 3 on p.~26]{MR717033}}]\label{tamepadic}
 Let $L/K$ be a Galois extension of $p$-adic fields with Galois group $G$. Then $\mathcal{O}_L$ is free of rank $1$ as an $\mathcal{O}_K[G]$-module if and only if the extension is tamely ramified.
\end{theorem}
More generally, one may ask whether $\mathcal{O}_L$ is free over other subrings of $K[G]$. It turns out that there is a natural candidate for a ring over which $\mathcal{O}_L$ could be free, the so-called \textbf{associated order}.

\begin{definition}
The associated order (of $\mathcal{O}_L$ in $K[G]$) is the ring
\[
\mathfrak{A}_{L/K} = \{ \lambda \in K[G] : \lambda \cdot x \in \mathcal{O}_L \; \forall x \in \mathcal{O}_L \}.
\]
\end{definition}
It is well-known that the associated order is indeed an $\mathcal{O}_K$-order in $K[G]$. Furthermore, if $\mathcal{O}_L$ is free of rank 1 over an $\mathcal{O}_K$-order $\Lambda$ in $K[G]$, then necessarily $\Lambda=\mathfrak{A}_{L/K}$ (see Proposition \ref{associatedfree}).
In our setting, we have that $\mathfrak{A}_{L/K}$ coincides with $ \mathcal{O}_K[G]$ precisely  when $L/K$ is at most tamely ramified.

In this paper, we investigate the situation when $\mathcal{O}_L$ is \emph{not} free over $\mathcal{O}_K[G]$. More specifically, we measure the failure of freeness by studying the minimal index of a free $\mathcal{O}_K[G]$-module inside $\mathcal{O}_L$, namely the quantity
\[
m(L/K) := \min_{\alpha \in \mathcal{O}_L} [\mathcal{O}_L : \mathcal{O}_K[G] \alpha],
\] 
where $[\mathcal{O}_L : \mathcal{O}_K[G]\alpha]$ denotes the group index. This quantity is well defined since for every integral normal basis generator $\alpha$ of the extension the index $[\mathcal{O}_L : \mathcal{O}_K[G]\alpha]$ is finite.
Clearly, by Theorem \ref{tamepadic}, we have $m(L/K)=1$ if and only if the extension is at most tamely ramified. Notice that, even though this is not clear from the definition, $m(L/K)$ is in principle computable for any finite extension $L/K$: this is the content of Theorem \ref{thm:mLKEffectivelyComputable}.
In certain settings it may be more natural to consider indices of the form $[\mathcal{O}_L : \mathfrak{A}_{L/K} \alpha]$, but, as we will see, these differ from $[\mathcal{O}_L : \mathcal{O}_K[G] \alpha]$ only by a constant factor independent of $\alpha$ (cf.~Proposition \ref{productindex}). On the other hand, the proportionality factor $[\mathfrak{A}_{L/K} : \mathcal{O}_K[G]]$ also carries interesting arithmetic content, so that $m(L/K)$ might in fact capture more complete information than the indices $[\mathcal{O}_L : \mathfrak{A}_{L/K} \alpha]$.

It is not the first time that the quantity $m(L/K)$ appears in the literature, in more or less explicit form. For example, in \cite{MR3411126} Johnston computed an explicit free generator of the ring of integers over its associated order in any wildly and weakly ramified Galois extension $L/K$ of local fields (for some earlier results  see also Burns \cite{MR1760494}); a crucial ingredient in his approach is precisely the computation of $m(L/K)$ in this situation.

\begin{theorem}[{\cite[Proof of Theorem 1.2]{MR3411126}}]\label{thm:JohnstonWeaklyRamified}
 Let $L/K$ be a wildly and weakly ramified Galois extension of $p$-adic fields. Then
$
  m(L/K)=p^{f_L}.
$
\end{theorem}
In a closely related direction, K\"ock \cite{MR2089083} and Johnston \cite{MR3411126} discuss whether a power of the maximal ideal $(\pi_L)$ can be $\mathcal{O}_K[G]$-free, thus giving some information on the possible indices $[\mathcal{O}_L : \mathcal{O}_K[G]\alpha]$; see also Ullom \cite{MR263790}.
Besides being useful for understanding the \textit{additive} structure of $\mathcal{O}_L$, free $\mathcal{O}_K[G]$-submodules $M$ of the ring of integers also appear in the cohomological study of local class field theory. In fact, starting from such an $M$, one can construct a cohomologically trivial submodule $V$ of finite index in $\mathcal{O}_L^\times$ (see for example \cite[p.~8]{MR2467155}); the cohomology of $\mathcal{O}_L^\times$ is then isomorphic to that of the finite quotient $\mathcal{O}_L^\times/V$, so that information on $[\mathcal{O}_L:M]$ (hence on $[\mathcal{O}_L^\times : V]$) translates into bounds for the cohomology groups $H^i(G, \mathcal{O}_L^\times)$.

In this paper we provide both widely-applicable estimates for $m(L/K)$ and exact formulas in several special cases.
Our first result is the following  general upper bound for  $m(L/K)$, that will be proved in Subsection \ref{sect:GeneralBound}. Notice that $m(L/K)$ is by definition a power of $p$, so bounding its absolute value is equivalent to bounding its $p$-adic valuation.
\begin{theorem}\label{thm:IntroGeneralBound}
 Let $L/K$ be a Galois extension of $p$-adic fields. Then
\[
v_p (m(L/K)) \leq f_L(e_{L/K}-1) + \frac{1}{2} [L:\mathbb{Q}_p]  \cdot v_p([L:K]).
\]
\end{theorem}
We  note that the bound of the previous theorem can be relaxed to an estimate depending only on $[L:\mathbb{Q}_p]$.

 Much more is known on the Galois structure of $L/K$  when $L$ is absolutely abelian. 
As a consequence, in this case we obtain an explicit formula for the value of $m(L/K)$. 
\begin{theorem}\label{thm:IntroAbsolutelyAbelian}
Let $L/K$ be a Galois extension of $p$-adic fields, with $p$ odd, and assume that $L/\Q_p$ is abelian. Denote by $L^{nr}$ the maximal unramified subextension of $L/K$. Then
$$
\begin{array}{rl}
 v_p( m(L/K))&=v_{p}(m(L/L^{nr}))\\
  &  \displaystyle={\frac{f_L}{2}\left(e_{L}\cdot v_p(e_{L/K})- \sum_{d \mid e_{L/K}} \frac{\varphi(d)}{[L^{nr}(\zeta_d):L^{nr}]} v_{L^{nr}}( \operatorname{disc}(L^{nr}(\zeta_d)/L^{nr}))\right)}.
\end{array}
$$
 \end{theorem}
We prove Theorem~\ref{thm:IntroAbsolutelyAbelian} in Section~\ref{sec:assab}.

In Section~\ref{sec:local-global} we consider an analogous quantity $m(E/F)$ for a number field extension $E/F$ and, in the case when $F=\mathbb{Q}$ and $E/\mathbb{Q}$ is abelian,  we establish some relations between the $p$-adic valuation of the minimal index for
$E/\mathbb{Q}$, and the  $p$-adic valuation of the minimal index for their $p$-adic completions.
 
Section~\ref{sec:degp} is devoted to the special interesting case of cyclic extensions of  degree $p$. We prove the following exact formula for $m(L/K)$, which holds  without any assumption on $K$. 
\begin{theorem}\label{thm:IntroCyclicDegreep}
Let $L/K$ be a ramified Galois extension of $p$-adic fields of degree $p$, with ramification jump $t$. Let $a \in \{0,\ldots,p-1\}$ be the residue class of $t$ modulo $p$ and set $\nu_i = \left\lfloor \frac{a+it}{p} \right\rfloor$.
Then if $a \neq 0$ we have
$
v_p( m(L/K))= f_K \left(\sum_{i=0}^{p-1} \nu_i+\min_{0\leq i\leq p-1}(i e_K-(p-1)\nu_i) \right),
$
while for $a=0$ we have $v_p(m(L/K))=\frac{1}{2}[L:\mathbb{Q}_p]$.
\end{theorem}
The method used to show Theorem \ref{thm:IntroCyclicDegreep} also allows us to give, in Section~\ref{sect:BertrandiasFerton}, a  new proof of a result originally due to Ferton and Bertrandias \cite[Théorème]{MR296047}.
\begin{theorem}\label{thm:IntroFertonBertrandias}
Let $L/K$ be a totally ramified cyclic extension of degree $p$ of $p$-adic fields. Let $t$ be the unique ramification jump of the extension (in the lower numbering) and let $a \in \{0,\ldots,p-1\}$ be the residue class of $t$ modulo $p$. The following hold:
\begin{enumerate}
\item if $a=0$ or $a \mid p-1$, then $\mathcal{O}_L$ is free over $\mathfrak{A}_{L/K}$. 
\item Suppose that the inequality $t < \frac{ep}{p-1} - 1$ holds. If $\mathcal{O}_L$ is free over $\mathfrak{A}_{L/K}$, then $a \mid p-1$.
\end{enumerate}
\end{theorem}

\begin{remark}
The inequality $t < \frac{ep}{p-1} - 1$ corresponds to the condition that $L/K$ is not \textit{almost-maximally ramified}, see Definition \ref{def:AlmostMaximallyRamified}.
\end{remark}

Finally, a natural related question is to understand which elements generate $\mathcal{O}_K[G]$-submodules of minimal index. We obtain some partial results on this problem by exhibiting an explicit minimal element for cyclic extensions of degree $p$ (Proposition \ref{prop:aeq0} and Remark \ref{rmk:MinimalElement}).

\subsection*{Acknowledgements}
The authors are indebted to Henri Johnston for his many insightful comments, especially on Theorem \ref{thm:IntroGeneralBound}, and to Nigel Byott for helpful discussions and for pointing out some relevant results in the literature.  We are also grateful to Cornelius Greither for his comments on a preliminary version of the paper and for suggesting Remark \ref{rmk:greither}.

\section{Preliminaries}\label{sect:Preliminaries}
\subsection{Preliminaries on discriminants}\label{sub:discriminants}
We will use the notation fixed in the Introduction, so $L/K$ is  a Galois extension of $p$-adic fields with Galois group $G$.
The normal basis theorem ensures that $L/K$ admits a normal basis, namely a $K$-basis of the form 
$\{\sigma(\alpha)\}_{\sigma\in G}$ for some $\alpha\in L$ (equivalently, we have $L=K[G]\alpha$). In this case the element $\alpha$ is called a normal basis generator. Clearly, $\alpha$ can be chosen to lie in $\mathcal{O}_L$, and it is easy to see that $\alpha \in \mathcal{O}_L$ is a normal basis generator if and only if $\mathcal{O}_K[G]\alpha$ is a \emph{full} $\mathcal{O}_K$-lattice in $L$;
this happens precisely when $[\mathcal O_L:\mathcal{O}_K[G]\alpha]$ is  finite. Therefore, in order to compute $m(L/K)$, it suffices to consider elements $\alpha \in \mathcal{O}_L$ that are normal basis generators.

Let $X\supseteq Y$ be $\mathcal O_K$-lattices of the same rank ($\mathcal O_K$ is a principal ideal domain, so the rank of an $\mathcal O_K$-lattice, which is free, is well defined). Choosing two bases for $X$ and $Y$, we easily see that the quotient $X/Y$ is of the form $\oplus_{i=1}^r \mathcal O_K/(\pi_K^{a_i})$, hence in particular finite, and we define the \emph{module index} $[X:Y]_{\mathcal O_K}$ as the ideal of $\mathcal O_K$ generated by $\pi_K^{\sum_{i=1}^r a_i}$. The \emph{subgroup index} is instead defined as $[X:Y]:=|X/Y|$, and we have
\[
[X:Y]=N([X:Y]_{\mathcal O_K}),
\]
where $N$ is the ideal norm. 
More generally, if $V$ is a finite-dimensional $K$ vector space and $X$ and $Y$ are full $\mathcal{O}_K$-lattices in $V$, then we set
$[X:Y]_{\mathcal O_K}=[X:X\cap Y]_{\mathcal O_K}[Y:X\cap Y]^{-1}_{\mathcal O_K}$, and this is a fractional ideal of $\mathcal O_K$.
Note that there exists an invertible $K$-linear transformation of $V$ such
that $\varphi(X)=Y$, so that 
\begin{equation}
\label{eq:ind-det}
[X:Y]_{\mathcal O_K}=(\det \varphi)\mathcal O_K
\end{equation}
(see \cite{CF}  for details and \cite{DCDJA} for an overview). 

\begin{remark}\label{remark:nodvr}
Let $L/K$ be an extension of number fields and let $X, Y$ be full $\mathcal{O}_K$-lattices in $L$. In this setting, the index $[X:Y]_{\mathcal{O}_K}$ is defined to be the only ideal of $\mathcal{O}_K$ with the following property: for every maximal ideal $\mathfrak{p}$ of $\mathcal{O}_K$, the $\mathfrak{p}$-adic valuation of $[X:Y]_{\mathcal{O}_K}$ is equal to that of $[X\otimes_{\mathcal{O}_K}\mathcal{O}_{K_\mathfrak p}:Y\otimes_{\mathcal{O}_K}\mathcal{O}_{K_\mathfrak p}]_{\mathcal{O}_{K_\mathfrak p}}$, where $K_\mathfrak p$ is the completion of $K$ at $\mathfrak{p}$. Note that a formula analogous to \eqref{eq:ind-det} holds if $X$ and $Y$ are free over $\mathcal{O}_K$, which happens for instance whenever $K$ has class number $1$.
\end{remark}

Suppose now that $X$ is an $\mathcal O_K$-lattice and let $V := X \otimes_{\mathcal{O}_K} K$ be the corresponding finite-dimensional vector space. Let $B$ be a non-degenerate symmetric $K$-bilinear form on $V$. 
\begin{definition}\label{def:Discriminant}
We define the \emph{discriminant} of $X$ by means of the module index:
$$\disc_{B}(X):=[X':X]_{\mathcal O_K},$$
where $X'$ is  the dual module of $X$ with respect to $B$, namely 
$$X'=\{v\in V\mid B(v,X)\subseteq \mathcal O_K\}.$$
\end{definition}
With these definitions, for $Y\subseteq X$ we have the following proposition, which is a special case of \cite[Proposition 4 on p.12]{CF}.
\begin{proposition}\label{prop:disc} Let $X,Y$ be $\mathcal O_K$-lattices such that $ X \otimes_{\mathcal{O}_K} K = Y \otimes_{\mathcal{O}_K} K=:V$. Let $B$ be a non-degenerate symmetric bilinear form on $V$.
Then 
\begin{enumerate}
\item $\disc_{B}(Y)=[X:Y]_{\mathcal O_K}^2\disc_{B}(X).$
\item If $X$ is the free $\mathcal O_K$-module generated by $\{x_i\}_{i=1}^n$ then 
\[
\disc_K(X)=\det\{B(x_i,x_j)\}.
\]
\end{enumerate}
\end{proposition}
\subsection{Preliminaries on ramification groups}
Let $L/K$ be a Galois extension of $p$-adic fields with Galois group $G$. %
\begin{definition}\label{def:RamificationGroups}
 For $i \geq -1$ we let
\[
G_i = \{g \in G : (g-1)(\mathcal{O}_L) \subseteq (\pi_L)^{i+1} \}
\]
be the $i$-th ramification group of $G$ in the lower numbering.
\end{definition}

Clearly, $G_{i+1}\unlhd G_i$ for each $i\ge-1$ and 
we say that $t$ is a (lower) ramification jump for the extension $L/K$  if $G_t\ne G_{t+1}$. It is well-known that for a fixed Galois extension $L/K$ the group $G_i$ is trivial for $i$ large enough (see \cite[Chapter 4]{MR554237} for details). The extension $L/K$ is called \textit{unramified} (respectively \textit{tamely ramified}, \textit{weakly ramified}) when $G_0$ (respectively $G_1, G_2$) is trivial.
The unique ramification jump of a ramified extension of degree $p$ has the following property.
\begin{proposition}[{\cite[Chapter IV, §2, Ex.~3]{MR554237} and \cite[III.2, Prop.~2.3]{MR1915966}}]\label{prop:InequalitiesOneAndt}
Let  $L/K$ be a cyclic ramified extension of degree $p$ and let $t$ be its ramification jump. Then:
\begin{enumerate}
\item the inequality $1 \leq t \leq \frac{e_Kp}{p-1}$ holds;
\item if $p \mid t$, then $t = \frac{e_Kp}{p - 1}$, the ground field $K$ contains the $p$-th
roots of unity, and there exists a uniformiser $\pi_K$ of $K$ such that $L = K(\pi_K^{1/p})$.
\end{enumerate}
\end{proposition}

\begin{definition}\label{def:AlmostMaximallyRamified}
Let $L/K$ be a ramified cyclic  extension of degree $p$. The extension $L/K$ is called \textit{maximally ramified} if its ramification jump assumes its maximum possible value, namely $t=\frac{e_Kp}{p-1}$, and is called \textit{almost-maximally ramified} if $t$ satisfies
\[
\frac{e_Kp}{p-1}-t\le1 . 
\]

More generally, a totally ramified cyclic extension $L/K$ of degree $dp^n$, with $(d,p)=1$, is  called \textit{almost-maximally ramified} if its first positive ramification jump $t_1$ satisfies 
\[
\frac{e_Kdp}{p-1}-t_1\le 1.
\]
\end{definition}

\begin{remark}
Jacobinski \cite{jac64} defined an extension $L/K$ with Galois group $G$ to be almost-maximally ramified if all idempotents $e_H=\frac1{|H|}\sum_{\sigma\in H}\sigma$  belong to the associated order $\mathfrak{A}_{L/K}$,  when $H$ ranges over all subgroups of $G$ included between two consecutive ramification groups of the extension. Using basic facts on ramification explained in \cite[Chapters 3 and 4]{MR554237}, one can verify that $e_H\in\mathfrak{A}_{L/K}$ if and only if
\begin{equation}\label{eq:RamificationInequality}
 \sum_{i=0}^\infty (|G_i(L/L^H)|-1)\geq e_Lv_p(|H|),
\end{equation}
where $L^H$ is the field fixed by $H$ and $G_i(L/L^H)$ the $i$-th ramification group of the extension $L/L^H$.
One may then show that Jacobinski's definition is equivalent to the inequalities in Definition \ref{def:AlmostMaximallyRamified}.
See \cite[$\S$1.2]{MR513880} or \cite[Proposition 1]{MR543208} for further details.
\end{remark}

\section{General results}
Let $L/K$ be a Galois extension of $p$-adic fields with group $G$.
As observed in the Introduction, the quantity $m(L/K)=\min_{\alpha \in \mathcal{O}_L} [\mathcal{O}_L : \mathcal{O}_K[G]  \alpha]$ is well defined as a minimum, and can be computed on the integers which are normal basis generators. In fact, these are precisely  the cases for which  $\mathcal{O}_K[G]  \alpha$ is a full $\mathcal{O}_K$-submodule of L. 

We are interested in bounding the value of  $m(L/K)$ in terms of the standard invariants of the extensions $L/K$ and $L/\Q_p$. Before doing this, we will show in Theorem \ref{thm:mLKEffectivelyComputable} that in principle the minimal index $m(L/K)$ can be effectively computed for any finite Galois extension $L/K$. 
In practice, direct computation of $m(L/K)$ with the algorithm we propose is limited to very simple extensions only. For this reason, an \textit{a priori} bound on $m(L/K)$ (such as in Theorem \ref{thm:IntroGeneralBound}) can be very useful both for the study of a single extension and for understanding the behaviour of $m(L/K)$ in families.%

\subsection{The minimal index $m(L/K)$ is effectively computable}

\begin{theorem}\label{thm:mLKEffectivelyComputable}
Let $L/K$ be a finite Galois extension of $p$-adic fields with Galois group $G$. The quantity $m(L/K)$ may be determined through a finite, effective procedure.
\end{theorem}
\begin{proof}
We fix a basis $\alpha_1,\ldots,\alpha_n$ of $\mathcal{O}_L$ over $\mathcal{O}_K$ and use it to identify $L$ with $K^n$ and $\mathcal{O}_L$ with $\mathcal{O}_K^n$:
for a vector $v=(c_{1},\ldots,c_{n}) \in \mathcal{O}_K^n$, we let $\omega(v) \in \mathcal{O}_L$ be the element $\sum_{i=1}^n c_i \alpha_i$.
Note that an integral basis may be computed effectively, for example using the Montes algorithm \cite{MR3105943, MR3276340}. 

Let $g_1=\operatorname{Id},\ldots,g_n$ be an enumeration of the elements of $G$. Every $g_i$ corresponds to a $K$-linear transformation of $K^n$, which can be represented by a matrix $M_i \in \operatorname{Mat}_{n}(K)$. Notice that  $M_k=\{m_{ij}^{(k)}\}$ where $g_k(\alpha_j) = \sum_{i=1}^n m_{ij}^{(k)} \alpha_i$. Since the $\alpha_j$'s are a basis of $\mathcal{O}_L$ over  $\mathcal{O}_K$ it follows that  the entries of $M_k$ lie in $\mathcal{O}_K$.

Let $\omega_1, \omega_2 \in \mathcal{O}_L$ and let $v_1, v_2$ be the vectors of the coordinates of $\omega_1,\omega_2$ in the basis $\alpha_j$. Assume that   $\omega_1$ generates a normal basis and let $[\mathcal{O}_L : \mathcal{O}_K[G]\omega_1]_{\mathcal{O}_K} = \pi_K^{R}\mathcal{O}_K$.   With the previous notation,  for $i=1,2,$  we have 
\begin{equation}
\label{indici}
[\mathcal{O}_L : \mathcal{O}_K[G] \omega_i] = N\left( \det \left( v_i \bigm\vert M_2v_i \bigm\vert \cdots \bigm\vert M_n v_i \right) \right).
\end{equation}

We claim that if  $v_2 \equiv v_1 \pmod{\pi^{R+1}}$ then 
\[
[\mathcal{O}_L : \mathcal{O}_K[G]   \omega_2]_{\mathcal{O}_K} =[\mathcal{O}_L : \mathcal{O}_K[G]   \omega_1]_{\mathcal{O}_K}:
\]
indeed from $v_2 \equiv v_1 \pmod{\pi_K^{R+1}}$ we obtain $M_k v_1 \equiv M_kv_2 \pmod{\pi_K^{R+1}}$ for every $k$, so the claim follows from Equation \eqref{indici}.

We turn to the description of a possible algorithm to compute $m(L/K)$. First we claim that one can effectively find a normal basis generator $\omega_0$ of $L/K$. The usual proof of the normal basis theorem (see for example \cite[Theorem VI.13.1]{MR1878556}) shows that there exists a non-zero (and effectively computable) polynomial $p(x_1,\ldots,x_a) \in K[x_1,\ldots,x_a]$ such that all elements $\omega=\sum_{i=1}^a c_i\alpha_i$ in $L$ that do \textit{not} generate a normal basis for $L/K$ satisfy $p(c_1,\ldots,c_a)=0$. Hence it suffices to find a point where $p(x_1,\ldots,x_n)$ does not vanish, and it is enough to try sufficiently many values of $x_1,\ldots,x_n$ to find some $v_0:=(c_1,\ldots,c_n) \in \mathcal{O}_K^n$ for which $p(c_1,\ldots,c_n) \neq 0$. The element $\omega_0 := \omega(v_0) \in \mathcal{O}_L$ is then a normal basis generator for $L/K$.
Next we compute the positive integer $R_0$ defined by the equality $[\mathcal{O}_L : \mathcal{O}_K[G]   \omega_0]=\pi_K^{R_0}\mathcal{O}_K$. This can also be done effectively: indeed, it suffices to compute the determinant
\[
\det\left( v_0 \bigm\vert M_2v_0 \bigm\vert \cdots \bigm\vert M_nv_0 \right)
\]
to sufficient $\pi_K$-adic precision to ensure that one can determine its $\pi_K$-adic valuation.

Fix representatives $v_1,\ldots,v_s$ of the (finitely many) residue classes in $\mathcal{O}_K^n/(\pi_K^{R_0+1} \mathcal{O}_K)^{n}$, and for every $i=1,\ldots,s$ let $\omega_i=\omega(v_i)$. Write $[\mathcal{O}_L : \mathcal{O}_K[G]   \omega_i]=\pi_K^{R_i} \mathcal{O}_K$. By the same argument as above, each of the finitely many $R_i$ is effectively computable, and there is no loss of generality in assuming that there is an index $i_0$ for which $\omega_{i_0}=\omega_0$.
We claim that $m(L/K)$ is equal to the norm of the ideal $\pi_K^{\min_i R_i}$.
To prove this, let
\[
r_i = \min_{v \equiv v_i \bmod{\pi_K^{R_0+1}}} v_K [\mathcal{O}_L : \mathcal{O}_K[G]   \omega(v)]_{\mathcal{O}_K},
\]
where the minimum is taken over all vectors $v \in \mathcal{O}_K^n$ all of whose coordinates are congruent to the respective coordinates of $v_i$ modulo $\pi_K^{R_0+1}$.
We now show that either $r_i=R_i$ or $r_i \geq R_0$. 
Indeed,  it is clear by definition that $r_i \leq R_i$. So $r_i=R_i$ or  $r_i < R_i$, and in this last case we can show that  $r_i \geq R_0$. In fact, if we had $r_i < R_0$ the index would be constant on the coset $v_i +\pi_K^{R_0+1}$, and this is not the case since we are assuming $r_i< R_i$. 
This proves that for all $i$ we have
\[
\min\{r_i, R_0\} = \min\{R_i, R_0 \};
\]
note in particular that for the index $i_0$ such that $\omega_i=\omega_0$ we have $r_{i_0}=R_{i_0}=R_0$.
Taking the minimum over all $i$ now yields
\[
\begin{aligned}
\min_{\omega \in \mathcal{O}_L} v_K[\mathcal{O}_L : \mathcal{O}_K[G]   \omega]_{\mathcal{O}_K} %
& = \min_i r_i = \min_i \min\{r_i, r_{i_0}\} \\
& = \min_i \min\{R_i, R_0\} = \min_i R_i,
\end{aligned}
\]
hence $m(L/K) = N\left( \pi_K^{\min_i R_i} \right)$. Since the finitely many quantities $R_i$ are effectively computable, so is $m(L/K)$, as claimed.
\end{proof}

\subsection{The associated order}
 We briefly review the notion of \textit{associated order} in the general setting of Dedekind domains that we will need later. We also show that the quantity $m(L/K)$ may be expressed in a natural way in terms of the associated order. The following result is well known.
\begin{proposition}\label{associatedfree}
Let $R$ be a Dedekind domain with field of fractions $K$, and let $A$ be a finite-dimensional $K$-algebra. Let $N$ be an $R$-lattice such that $K  N:=K \otimes_{R} N \cong A$ as $A$-modules. Finally, suppose that $\Lambda$ is an $R$-order in $A$ such that $N$ is a free $\Lambda$-module. Then $\Lambda$ is equal to the associated order of $N$ in $A$:
\[
 \Lambda=\mathfrak{A}(A, N) := \{ \lambda \in  A \mid \lambda N \subseteq N \}.
\]
\end{proposition}
\begin{proof}
By assumption there exists $\alpha\in N$ such that $N=\Lambda \alpha$ is a free $\Lambda$-module. Then $\alpha$ is also such that $K  N=A \alpha$ is free over $A$. The inclusion $\Lambda\subseteq \mathfrak{A}(A, N)$ is clear. Conversely, let $a\in \mathfrak{A}(A, N)$; then $a\alpha\in N=\Lambda \alpha$, so there exists $\lambda\in\Lambda$ such that $a\alpha =\lambda\alpha$ in $A$. Since $K  N$ is freely generated by $\alpha$, this means that $a=\lambda\in \Lambda$.
\end{proof}

\begin{proposition}\label{productindex}
Let $L/K$ be a Galois extension of $p$-adic fields with Galois group $G$. Then 
\[
m(L/K) = [\mathfrak{A}_{L/K}:\mathcal{O}_K[G]] \cdot \min_{\alpha \in \mathcal{O}_L} [\mathcal{O}_L : \mathfrak{A}_{L/K}   \alpha].
\] 
In particular, $\mathcal{O}_L$ is free over $\mathfrak{A}_{L/K}$ if and only if $m(L/K) = [\mathfrak{A}_{L/K}:\mathcal{O}_K[G]]$. In this case, given an element $\beta$ that realises the minimal index $m(L/K)$, the elements realising $m(L/K)$ are precisely those of the form $\lambda \beta$ for $\lambda \in \mathfrak{A}_{L/K}^*$. Equivalently, an element $\beta \in \mathcal{O}_L$ realises the minimum if and only if it generates $\mathcal{O}_L$ over $\mathfrak{A}_{L/K}$.
\end{proposition}
\begin{proof}
Let $\beta\in \mathcal{O}_L$ be an element such that
\[
m(L/K) = [\mathcal{O}_L : \mathcal{O}_K[G]   \beta].
\] 
By the definition of the associated order, we have that $\mathfrak{A}_{L/K}   \beta\subseteq \mathcal {O}_L$, so 
\[
 m(L/K) = [\mathcal{O}_L : \mathfrak{A}_{L/K}   \beta]\cdot [\mathfrak{A}_{L/K}   \beta : \mathcal{O}_K[G]   \beta].
\]
Since 
$\beta$ generates a normal basis for $L/K$, we have that the $K$-linear map  $\varphi_\beta \colon K[G]\to K[G] \beta$ defined by $x\mapsto x\beta$ is nonsingular,  so $[\mathfrak{A}_{L/K}   \beta : \mathcal{O}_K[G]   \beta]=[\mathfrak{A}_{L/K}: \mathcal{O}_K[G]]$. The result follows.
\end{proof}

\begin{remark}\label{rmk:SeparateIndices}
Proposition \ref{productindex} shows that every index $[\mathcal{O}_L : \mathcal{O}_K[G]   \alpha]$ differs from the corresponding index $[\mathcal{O}_L : \mathfrak{A}_{L/K}   \alpha]$ by a constant independent of $\alpha$.
Combined with Proposition \ref{associatedfree}, this suggests that it may be more natural to study separately the two quantities $[\mathfrak{A}_{L/K} : \mathcal{O}_K[G]]$ and $\min_{\alpha \in \mathcal{O}_L} [\mathcal{O}_L : \mathfrak{A}_{L/K}   \alpha]$.
Therefore our results can be also read as bounds on the minimum of $[\mathcal{O}_L : \mathfrak{A}_{L/K}   \alpha]$.
\end{remark}

\subsection{Index of $\mathcal{O}_K[G]$ in a maximal order}\label{indexmaximal}
In the spirit of Remark \ref{rmk:SeparateIndices} we now give a simple upper bound on $[\mathfrak{A}_{L/K} : \mathcal{O}_K[G]]$ by estimating the index of $\mathcal{O}_K[G]$ in a \textit{maximal} $\mathcal{O}_K$-order $\mathfrak{M}$ in which it is contained.

We consider on $K[G]$ the bilinear form $(x,y) \mapsto \operatorname{tr}(xy)$, where we denote by $\operatorname{tr}(z)$ the trace of multiplication by $z$ on the $K$-vector space $K[G]$; it is nondegenerate since $K$ has characteristic $0$, so it allows us to compute discriminants of lattices in the sense of Definition \ref{def:Discriminant}. 
Let $\underline{m} = [\mathfrak{M} : \mathcal{O}_K[G]]_{\mathcal{O}_K}$. By Proposition \ref{prop:disc} we have an equality of ideals
\begin{equation}\label{eq:DiscIndex}
\underline{m}^2 \cdot \operatorname{disc} \mathfrak{M} = \operatorname{disc} \mathcal{O}_K[G],
\end{equation}
and writing $\underline{m} = \pi_K^A \mathcal{O}_K$ we obtain
\begin{equation}\label{eq:IndexIdealNumber}
[\mathfrak{M} : \mathcal{O}_K[G]] = N(\underline{m})= p^{f_KA}.
\end{equation}
We now show that $\operatorname{disc} \mathcal{O}_K[G]=|G|^{|G|} \mathcal{O}_K$. We claim that the trace of an element $g \in G \subset \mathcal{O}_K[G]$ is 0 if $g \neq \operatorname{Id}$ and is $|G|$ for $g=\operatorname{Id}$. To see this, notice that an $\mathcal{O}_K$-basis of $\mathcal{O}_K[G]$ is given by $\{g \in G\}$, and in this basis the left multiplication action by $g$ is represented by a permutation matrix, whose trace equals the number of fixed points of the map $h \mapsto gh$. The group axioms imply that either $h \mapsto gh$ has no fixed points ($g \neq \operatorname{Id}$), or every element of $G$ is fixed ($g=\operatorname{Id}$). This proves the claim.
In particular, the Gram matrix of the bilinear form $(x,y) \mapsto \operatorname{tr}(x,y)$ in this basis is
\[
\left( \operatorname{tr}(gh) \right)_{g, h \in G} = |G| \left(  \delta_{gh,\operatorname{id}} \right)_{g, h \in G},
\]
whose determinant is clearly $\pm |G|^{|G|}$, because $\left(  \delta_{gh,\operatorname{id}} \right)_{g, h \in G}$ is a permutation matrix.
From Equation \eqref{eq:DiscIndex} we obtain
\[
\underline{m}^2 = \frac{|G|^{|G|} \mathcal{O}_K}{\operatorname{disc} \mathfrak{M}} \bigm\vert |G|^{|G|} \mathcal{O}_K,
\]
so (writing again $\underline{m}=\pi_K^{A}\mathcal{O}_K$) we have $A \leq\frac12|G| e_K v_p(|G|)$. Via Equation \eqref{eq:IndexIdealNumber}, this implies 
\begin{equation}\label{eq:IndexMaximalOrderBound}
v_p\left(\#\frac{\mathfrak{M}}{\mathcal{O}_K[G]}\right) \leq \frac{1}{2} [K:\mathbb{Q}_p] \cdot |G| \cdot v_p(|G|).
\end{equation}
We summarise these results in the following Proposition.
\begin{proposition}\label{prop:MaximalOrder}
Let $\mathfrak{M}$ be a maximal order of $K[G]$. We have $[\mathfrak{M} : \mathcal{O}_K[G]]_{\mathcal{O}_K}=\pi_K^{A}\mathcal{O}_K$ with $A=\frac{1}{2}|G|e_K v_p(|G|) -\frac{1}{2} v_K(\operatorname{disc} \mathfrak{M} ) \leq \frac{1}{2}|G|e_K v_p(|G|)$.
\end{proposition}

\subsection{Proof of Theorem \ref{thm:IntroGeneralBound} }
\label{sect:GeneralBound}
In this section we prove a bound on $m(L/K)$ which holds in complete generality.
Let $\mathfrak{B}$ be an $\mathcal{O}_K$-order in $K[G]$ containing $\mathcal{O}_K[G]$. We define $\mathfrak{B}\mathcal{O}_L$ to be the $\mathfrak{B}$-lattice generated by $\mathcal{O}_L$, that is, 
\[
 \mathfrak{B} \mathcal{O}_L=\left\{\sum_i b_ix_i\right\}_{b_i\in \mathfrak{B},x_i\in\mathcal{O}_L}\subseteq L,
\]
and suppose that
$\mathfrak{B}  \mathcal{O}_L$ is a free $\mathfrak{B}$-module of rank $1$. We note that
 every maximal order $\mathfrak{B}$ has this property.  Indeed, let $\mathfrak{B}$ be a maximal order of the separable $K$-algebra $K[G]$; by the normal basis theorem $K\mathfrak{B}\mathcal{O}_L=L$ is  isomorphic to $K\mathfrak{B}=K[G]$ as a left $K[G]$-module, so $\mathfrak{B}\mathcal{O}_L$ is isomorphic to $\mathfrak{B}$ as a left $\mathfrak{B}$-module by \cite[Theorem 18.10]{MR1972204}.
\begin{lemma}
The set $\mathfrak{B} \mathcal{O}_L$ is a fractional ideal in $L$.
\end{lemma}
\begin{proof}
Indeed, given $\sum_{g \in G} k_g g$ in $\mathfrak{B}$ (where $k_g \in K$) and $x,y \in \mathcal{O}_L$ we have
\[
\left( \sum_{g \in G} k_g g x \right)   y = \sum_{g \in G} k_g g \left( x g^{-1}(y) \right) \in \mathfrak{B}\mathcal{O}_L.
\]
\end{proof}
In the light of this lemma we can write $\mathfrak{B}\mathcal{O}_L$ as $\pi_L^{-a}\mathcal{O}_L$ for a certain integer $a \geq 0$.
Let now $x$ be a generator of $\mathfrak{B}  \mathcal{O}_L$ over $\mathfrak{B}$ (in particular, $x$ is a normal basis generator for $L/K$), 
and let $0\neq r\in\mathcal O_K$ be such that $rx\in \mathcal{O}_L$.
Then we have:
\[\begin{split}
  [\mathfrak{B}  \mathcal O_L : \mathcal O_L]  [\mathcal O_L : \mathcal O_K[G] rx]&=[\mathfrak{B}  \mathcal O_L : \mathcal O_K[G] rx]\\
  &=[\mathfrak{B}x : \mathcal O_K[G] x]  [\mathcal O_K[G] x : \mathcal O_K[G] rx],
\end{split}
\]  
and since $x$ is a normal basis generator this quantity is equal to
\[
\begin{split} 
[\mathfrak{B} : \mathcal O_K[G] ]\cdot [\mathcal O_K[G]  : \mathcal O_K[G] r]   & =[\mathfrak{B} : \mathcal O_K[G] ]\cdot |\mathcal{O}_K/r\mathcal{O}_K|^{|G|} \\
  & =[\mathfrak{B} : \mathcal O_K[G] ]\cdot p^{{f_K|G|v_K(r)}},
 \end{split}
\]
where we have used that $\mathcal O_K[G]$ is a free $\mathcal{O}_K$-module of rank $|G|$. 
Hence we have
\[
m(L/K) \leq \left[  \mathfrak{B} : \mathcal{O}_K[G] \right] \cdot \frac{p^{{f_K|G|v_K(r)}}}{[ \mathfrak{B}   \mathcal{O}_L : \mathcal{O}_L]},
\]
and we may replace $[ \mathfrak{B}   \mathcal{O}_L : \mathcal{O}_L]$ by $[ \pi_L^{-a} \mathcal{O}_L : \mathcal{O}_L] = p^{a f_L}$.
Given that $\mathfrak{B}\mathcal{O}_L$ is the fractional ideal $(\pi_L^{-a})$, it is clear that $\pi_K^b x$ is in $\mathcal{O}_L$ provided that $v_L(\pi_K^b) \geq a$, that is,
$
b \geq \frac{a}{e_{L/K}}.
$ 
Plugging in $r=\pi_K^b$ (with $b=\lceil\frac{a}{e_{L/K}}\rceil$) in the formula above, and noticing that $b \leq \frac{a+e_{L/K}-1}{e_{L/K}}$, we obtain
\[
\begin{aligned}
m(L/K) & \leq \left[  \mathfrak{B} : \mathcal{O}_K[G] \right] \cdot \frac{ p^{f_K |G| b} }{p^{af_L}} \\
& \leq \left[  \mathfrak{B} : \mathcal{O}_K[G] \right] \cdot  p^{f_K f_{L/K}e_{L/K} \frac{a+e_{L/K}-1}{e_{L/K}} - af_L} \\
& = \left[  \mathfrak{B} : \mathcal{O}_K[G] \right] \cdot  p^{f_L (e_{L/K}-1)}.
\end{aligned}
\]
Using Equation \eqref{eq:IndexMaximalOrderBound} we then obtain
\[
\begin{aligned}
v_p (m(L/K)) & \leq f_L(e_{L/K}-1) + \frac{1}{2} [K:\mathbb{Q}_p] \cdot  |G| \cdot v_p(|G|) \\
& = f_L(e_{L/K}-1) +  \frac{1}{2} [L:\mathbb{Q}_p] v_p([L:K]).
\end{aligned}
\]
This proves Theorem \ref{thm:IntroGeneralBound}, and it is clear that one can get a very simple (albeit rough) estimate by replacing $f_L(e_{L/K}-1)$ with $[L:\mathbb{Q}_p]$, thus obtaining
\begin{equation}\label{eq:EasierBound}
v_p(m(L/K)) < [L:\mathbb{Q}_p] (1 + \frac{1}{2} v_p([L:K])).
\end{equation}
\begin{remark}\label{rmk:AlmostOptimalBound}
Comparison with cases in which we can compute $m(L/K)$ exactly suggests that the bound  of Theorem \ref{thm:IntroGeneralBound} is \emph{sharper} when  $L/K$ is highly ramified. In fact, in some maximally ramified cases our bound is almost optimal: take for example $K=\mathbb{Q}_p(\zeta_{p^r})$ and $L=K(\sqrt[p^r]{\pi_K})$. One may show that $v_p(m(L/K))=\frac{r}{2} [L:\mathbb{Q}_p]$ (see Remark \ref{rmk:KummerDegreepn}), while our result -- in the form of Equation \eqref{eq:EasierBound} -- gives the upper bound $\left(\frac{r}{2}+1\right) [L:\mathbb{Q}_p]$, which is essentially sharp for large $r$.
Theorem \ref{thm:IntroGeneralBound}
might thus be useful in situations where a complicated ramification structure prevents the use of other tools.
\end{remark}

\subsection{Reduction to the totally ramified case}\label{subsect:ReductionToTotallyRamifiedCase}
Let $L/K$ be a Galois extension of $p$-adic fields, $L^{nr}$ be its maximal unramified subextension, and $G_0$ be the inertia subgroup of $G=\operatorname{Gal}(L/K)$. In this section we show that $[\mathfrak{A}_{L/K} : \mathcal{O}_K[G]]$ is bounded above by the analogous quantity for the extension $L/L^{nr}$, and discuss some cases in which the equality $m(L/K)=m(L/L^{nr})$ holds.
To compare $\mathfrak{A}_{L/K}$ and $\mathfrak{A}_{L/L^{nr}}$ we start with a result of Jacobinski (see the beginning of $\S2.1$ in \cite{MR513880}), according to which we have
\[
 \mathfrak{A}_{L/K}=\bigoplus_{s\in G/G_0}(\mathfrak{A}_{L/L^{nr}}\cap K[G_0])s,
\]
where $G/G_0$ denotes a fixed system of left coset representatives. Then
\[\begin{split}
   [\mathfrak{A}_{L/K}:\mathcal{O}_K[G]]&=[\oplus_{s\in G/G_0}(\mathfrak{A}_{L/L^{nr}}\cap K[G_0])s:\oplus_{s\in G/G_0}\mathcal{O}_K[G_0]s]\\
   &=[\mathfrak{A}_{L/L^{nr}}\cap K[G_0]:\mathcal{O}_K[G_0]]^{[G:G_0]}\\
   &=[\mathcal{O}_{L^{nr}}\otimes_{\mathcal{O}_K}(\mathfrak{A}_{L/L^{nr}}\cap K[G_0]) :\mathcal{O}_{L^{nr}}[G_0]]
  \end{split} 
\]
(we are using that $\mathcal{O}_{L^{nr}}$ is free over $\mathcal{O}_K$ of rank $[G:G_0]$). As noted for instance by Berg\'e in \cite{MR513880}, we always have an injection
\begin{equation}\label{injection}
 \mathcal{O}_{L^{nr}}\otimes_{\mathcal{O}_K}(\mathfrak{A}_{L/L^{nr}}\cap K[G_0])\hookrightarrow \mathfrak{A}_{L/L^{nr}}.
\end{equation}
We have thus obtained the following.
\begin{proposition}\label{indexassociated}
In the above notation,
 \[
 [\mathfrak{A}_{L/K}:\mathcal{O}_K[G]]\leq[\mathfrak{A}_{L/L^{nr}}:\mathcal{O}_{L^{nr}}[G_0]],
\]
with equality if and only if (\ref{injection}) is a surjection.
\end{proposition}
However, the injection \eqref{injection} is not always a surjection. In the same article, Berg\'e computed the image of \eqref{injection} when $G_1$ is cyclic, see \cite[Remarque after Lemme 5]{MR513880}. She also gave an elegant description of the image of \eqref{injection} in the general case that we now recall.
\begin{proposition}[{\cite[Proposition 4]{MR747998}}]\label{imageberge} The following equality holds:
\begin{equation}\label{intersection}
 \mathcal{O}_{L^{nr}}\otimes_{\mathcal{O}_K}(\mathfrak{A}_{L/L^{nr}}\cap K[G_0])=\bigcap_{g\in G}g\mathfrak{A}_{L/L^{nr}}g^{-1}.
\end{equation}
In particular, if $G$ is abelian then (\ref{injection}) is a surjection.
\end{proposition}
Note that $\bigcap_{g\in G}g\mathfrak{A}_{L/L^{nr}}g^{-1}$ is an $\mathcal{O}_{L^{nr}}$-order in $L^{nr}[G_0]$ contained in $\mathfrak{A}_{L/L^{nr}}$, hence $\mathcal{O}_L$ has the structure of a module over it. It is natural to ask whether there is a relation between freeness of $\mathcal{O}_L$ over $\mathfrak{A}_{L/K}$ and over $\bigcap_{g\in G}g\mathfrak{A}_{L/L^{nr}}g^{-1}$ (which would imply $\bigcap_{g\in G}g\mathfrak{A}_{L/L^{nr}}g^{-1}=\mathfrak{A}_{L/L^{nr}}$ by Proposition \ref{associatedfree}). In full generality we only know the following.
\begin{theorem}[{\cite[Théorème]{MR747998}}]\label{bergeprojective}
Keeping the above notation, $\mathcal{O}_L$ is projective over $\mathfrak{A}_{L/K}$ if and only if it is projective over $\bigcap_{g\in G}g\mathfrak{A}_{L/L^{nr}}g^{-1}$.
\end{theorem}
Fortunately there are situations in which knowing that $\mathcal{O}_L$ is projective is sufficient to conclude that it is free, for example when the relevant orders are commutative. More precisely, the next proposition shows that in several cases the minimal index $m(L/K)$ is controlled purely by the totally ramified extension $L/L^{nr}$.
\begin{proposition}\label{unramifiedindex}
Assume that $G_0$ is abelian and that $\mathcal{O}_L$ is free over $\mathfrak{A}_{L/K}$. Then $\mathcal{O}_{L}$ is free over $\mathfrak{A}_{L/L^{nr}}$ and 
\[
m(L/K)=m(L/L^{nr})=[\mathfrak{A}_{L/K}:\mathcal{O}_K[G]]=[\mathfrak{A}_{L/L^{nr}}:\mathcal{O}_{L^{nr}}[G_0]].
\]
Conversely, if $G$ is abelian and $\mathcal{O}_{L}$ is free over $\mathfrak{A}_{L/L^{nr}}$, then $\mathcal{O}_L$ is free over $\mathfrak{A}_{L/K}$ and as before
\[
m(L/K)=m(L/L^{nr})=[\mathfrak{A}_{L/K}:\mathcal{O}_K[G]]=[\mathfrak{A}_{L/L^{nr}}:\mathcal{O}_{L^{nr}}[G_0]].
\]
\end{proposition}
\begin{proof}
For the forward implication, assume that the $\mathfrak{A}_{L/K}$-module $\mathcal{O}_L$ is free, hence also projective. By Theorem \ref{bergeprojective} we have that $\mathcal{O}_L$ is projective over $\bigcap_{g\in G}g\mathfrak{A}_{L/L^{nr}}g^{-1}$. The latter is a commutative order since $G_0$ is abelian, so it is ``clean'', that is, every projective $\bigcap_{g\in G}g\mathfrak{A}_{L/L^{nr}}g^{-1}$-lattice that spans the algebra $L^{nr}[G_0]$ is free (see \cite{MR0175950} or \cite[IX Corollary 1.5]{MR0283014}).
It follows that $\mathcal{O}_L$ is free over $\bigcap_{g\in G}g\mathfrak{A}_{L/L^{nr}}g^{-1}$, so the latter is the associated order of $\mathcal{O}_L$ in $L^{nr}[G_0]$, that is,
\[
 \bigcap_{g\in G}g\mathfrak{A}_{L/L^{nr}}g^{-1}=\mathfrak{A}_{L/L^{nr}}
\]
and $\mathcal{O}_L$ is free over $\mathfrak{A}_{L/L^{nr}}$.
By Proposition \ref{imageberge} we have that (\ref{injection}) is a bijection. Therefore, by Proposition \ref{indexassociated} and Proposition \ref{productindex}, we have
\[
 m(L/K)=[\mathfrak{A}_{L/K}:\mathcal{O}_K[G]]=[\mathfrak{A}_{L/L^{nr}}:\mathcal{O}_{L^{nr}}[G_0]]=m(L/L^{nr}).
\]
As for the converse, assume that $\mathcal{O}_L$ is free over $\mathfrak{A}_{L/L^{nr}}$, which coincides with the intersection $ \bigcap_{g\in G}g\mathfrak{A}_{L/L^{nr}}g^{-1}$ since $G$ is abelian. It follows from Theorem \ref{bergeprojective} that $\mathcal{O}_L$ is projective over $\mathfrak{A}_{L/K}$, which is commutative, hence clean. We obtain that $\mathcal{O}_L$ is free over $\mathfrak{A}_{L/K}$, and the remaining equalities are proved as above.
\end{proof}

\section{The case when the associated order is maximal}
\label{sec:assab}
 When $\mathfrak{A}_{L/K}$ is a maximal order in $K[G]$ the problem of studying $m(L/K)$ simplifies considerably. In this section we study this situation and prove Theorem \ref{thm:IntroAbsolutelyAbelian}. We start
with the following simple lemma.
\begin{lemma}\label{associatedmaximal}
 Let $L/K$ be a Galois extension of $p$-adic fields with Galois group $G$. Suppose that $\mathfrak{A}_{L/K}$ is a maximal $\mathcal{O}_K$-order 
 in $K[G]$. Then $\mathcal{O}_L$ is free over $\mathfrak{A}_{L/K}$ and $m(L/K)=
 [\mathfrak{A}_{L/K}:\mathcal{O}_K[G]]$.
\end{lemma}
\begin{proof}
 The first assertion follows from \cite[Theorem (18.10)]{MR1972204}, the second from Proposition \ref{productindex}.
\end{proof}
Under the assumption of Lemma \ref{associatedmaximal}, the discussion in $\S$\ref{indexmaximal} shows that the minimal index $m(L/K)$ is determined by the discriminant of a maximal order in $K[G]$. We now describe some interesting instances of this situation.

If $G$ is abelian, the Wedderburn decomposition of $K[G]$ is given by a product of cyclotomic extensions of $K$,
\[
 K[G]\cong \prod_{\gamma\in \Phi} K(\gamma),
\]
where $\Phi$ is the set of orbits of the characters of $G$ under the action of the absolute Galois group of $K$ and $K(\gamma)$ is the extension of $K$ generated by the image of any character in the orbit $\gamma$. In this case the unique maximal order is 
\[
 \mathfrak{M}\cong \prod_{\gamma\in \Phi} \mathcal{O}_{K(\gamma)},
\]
and using Proposition \ref{prop:disc} one gets
\[
 \operatorname{disc}\mathfrak{M}=\prod_{\gamma\in \Phi} \operatorname{disc}(K(\gamma)/K).
\]
Hence, by Proposition \ref{prop:MaximalOrder}, when $G$ is abelian we have
\begin{equation}
\label{formula-max-abeliano}
 v_p[\mathfrak{M} : \mathcal{O}_K[G]]=\frac{f_K}{2}\left(e_K|G|v_p(|G|)-\sum_{\gamma\in \Phi} v_K( \operatorname{disc}(K(\gamma)/K))\right).
\end{equation}
If in addition $\mathfrak{A}_{L/K}=\mathfrak{M}$, this is also the value of $m(L/K)$.
As a consequence, we can deal with the case of absolutely abelian extensions of $\mathbb{Q}_p$ when $p$ is odd. 
\begin{proof}[Proof of Theorem \ref{thm:IntroAbsolutelyAbelian}]
In this case $L$ is absolutely abelian, and the main result of \cite{MR1627831} shows that $\mathcal{O}_L$ is free over $\mathfrak{A}_{L/K}$. This, together with Proposition \ref{unramifiedindex}, implies $m(L/K)=m(L/L^{nr})$. 
Thus we can suppose that $L/K$ is totally ramified. By the local version of the Kronecker-Weber Theorem, when the extension is totally ramified we know that $L\subseteq K(\zeta_{p^n})$ for some $n$. Hence $G=\text{Gal}(L/K)$ is naturally isomorphic to a subquotient of $\text{Gal}(\Q(\zeta_{p^n})/\Q)$, which is cyclic since $p$ is odd. So $G$ is also cyclic, and the Wedderburn decomposition of $K[G]$ is
\[
 K[G]\cong \prod_{d \mid [L:K]} K(\zeta_d)^{\frac{\varphi(d)}{[K(\zeta_d):K]}}.
\]
By \cite[Theorem 1]{MR1627831} the associated order $\mathfrak{A}_{L/K}$ is the maximal order in $K[G]$. %
By Lemma \ref{associatedmaximal} and Equation \eqref{formula-max-abeliano}, we get
\[
 \begin{aligned}
 v_p(m(L/K)) &=v_p(m(L/L^{nr})) \\ & =\frac{f_{L}}{2}\left(e_K|G_0|v_p(|G_0|)-\sum_{d\mid e_{L/K}} \frac{\varphi(d)}{[{L^{nr}(\zeta_d):L^{nr}]}} v_{L^{nr}}( \operatorname{disc}(L^{nr}(\zeta_d)/L^{nr}))\right)
 \end{aligned}
\]
as required.
\end{proof}
\begin{remark}\label{rmk:CasepEqualsTwo}
The formula in the previous theorem needs to be modified for $p=2$ and $L/K$ ramified. The main difference is that  the associated order is not necessarily maximal in this case (see \cite{MR1627831}). 
In addition, since the Galois group $G$ may not be cyclic in this case, for the maximal order we get:
$$ v_2[\mathfrak{M} : \mathcal{O}_K[G]]=\frac{f_{L}}{2}\left(|G_0|v_2(|G_0|)e_K-\sum_{d\mid e_{L/K}} \frac{a(d)}{[{L^{nr}(\zeta_d):L^{nr}]}} v_{L^{nr}}( \operatorname{disc}(L^{nr}(\zeta_d)/L^{nr}))\right),
$$
where $a(d)$ is the number of element of order $d$ in $G$.
When $\mathfrak{A}_{L/K}$ is maximal, this is also  the value of $m(L/K)$, but $m(L/K)$ may be much smaller in general. A general formula may also be established in the case $p=2$ by following the description of $\mathfrak{A}_{L/K}$ given in \cite{MR1627831}, but we content ourselves with describing a case in which the result of Theorem \ref{thm:IntroAbsolutelyAbelian} does not hold for $p=2$.

Let $m\ge2$,  $L=\Q_2(\zeta_{2^m})$ and $K=\Q_2(\zeta_{2^m}+\zeta_{2^m}^{-1})$. Let $\sigma$ be a generator of the Galois group $G$ of $L/K$, which is cyclic of order 2.
We have $K[G]\cong K\oplus K$, and so $\mathfrak{M}\cong\mathcal{O}_K\oplus\mathcal{O}_K$. Via the same isomorphism, which can be taken to be the Chinese Remainder Theorem map $\alpha+\beta\sigma\to(\alpha+\beta,\alpha-\beta)$, one can check that 
\[
\mathcal{O}_K[G]\cong\{(a,b)\in\mathcal{O}_K\oplus\mathcal{O}_K\mid a-b\in 2\mathcal{O}_K\}.
\]  
By \cite[Proposition 3]{MR1627831} the associated order in this case is  the ring generated by $\frac{1-\sigma}{\pi_K}$ over $\mathcal{O}_K[G]$, and via the above isomorphism
\[
\mathfrak{A}_{L/K}\cong\{(a,b)\in\mathcal{O}_K\oplus\mathcal{O}_K\mid a-b\in \pi_K^{e_K-1}\mathcal{O}_K\}.
\]  
Since $\mathcal{O}_L$ is free over $\mathfrak{A}_{L/K}$ by \cite[Proposition 3]{MR1627831}, these explicit descriptions allow us to easily compute 
\[
m(L/K)=[\mathfrak{A}_{L/K}:\mathcal{O}_K[G]]=N(\pi_K)=2,
\]
while on the other hand
\[
[\mathfrak{M} : \mathcal{O}_K[G]]=2^{e_K}=2^{2^{m-2}}.
\]
Finally, we remark that in this case the value $m(L/K)$ may also be obtained easily from Theorem \ref{thm:IntroCyclicDegreep} (one may check that in our situation we have $t=1$ and $\nu_1=1$, hence $v_2(m(L/K))=f_K \cdot 1 = 1$). In addition, the description of the associated order can also be obtained from Theorem \ref{thm:StructureOfALK}.
\end{remark}
The discriminants of the cyclotomic extensions of $\Q_p$ are well-known, so from Theorem \ref{thm:IntroAbsolutelyAbelian} we easily deduce the following.
\begin{corollary}\label{cor:AbsolutelyAbelianCase}
 Let $L/K$ be an absolutely abelian extension of $p$-adic fields with ramification index $p^nd$, where $p$ is an odd prime and $(d,p)=1$, and suppose that $K$ is unramified over $\Q_p$. Then
\[
  m(L/K)=p^{\frac{f_Ld(p^n-1)}{p-1}}.
\]
\end{corollary}
\begin{proof}
Up to replacing $K$ with $L^{nr}$ we can assume that $L$ is totally ramified over $K$ of degree $p^nd$, so from Theorem \ref{thm:IntroAbsolutelyAbelian} we have
\begin{equation}\label{eq:thFormula}
v_p(m(L/K))=\frac{f_L}{2}\left(dp^n n-\sum_{\delta\mid dp^n } \frac{\varphi(\delta)}{[K(\zeta_\delta):K]} v_{p}( \operatorname{disc}(K(\zeta_\delta)/K))\right).
\end{equation}
Let $s$ be a positive integer relatively prime to $p$ and  let $f_s= [\Q_p(\zeta_{s}):\Q_p].$ We recall that the extension $\Q_p(\zeta_{s})/\Q_p$ is unramified, so its discriminant is trivial,
while for each $r\ge1$ we have
\begin{equation}
\label{eq:disc}
 \operatorname{disc}(\Q_p(\zeta_{p^rs})/\Q_p )=p^{p^{r-1}(pr-r-1)f_s}
\end{equation}
(this follows e.g.~from \cite[Theorem 2.20]{MR2078267}, combined with the transitivity of the discriminant in towers). The extension $K/\Q_p$ is unramified, so $K\cap\Q_p(\zeta_{p^rs})= K\cap\Q_p(\zeta_{s})$; denote this extension by $K_s$. Letting $[K_s:\Q_p]=l_s$, the  extensions $K/K_s$ and $K(\zeta_{p^rs})/\Q_p(\zeta_{p^rs})$ have degree $\frac{f_K}{l_s}$ and, being unramified,  have trivial discriminant. We also have 
\[
[K(\zeta_{p^rs}):K]=[\Q_p(\zeta_{p^rs}):K_s]=\varphi(p^r)\frac{f_s}{l_s}.
\]
By transitivity of the discriminant in towers of extensions we get
$$ \operatorname{disc}(\Q_p(\zeta_{p^rs})/\Q_p )=N_{K_s/\Q_p}(\operatorname{disc}(\Q_p(\zeta_{p^rs})/K_s ))=\operatorname{disc}(\Q_p(\zeta_{p^rs})/K_s )^{l_s}.$$
Computing $\operatorname{disc}(K(\zeta_{p^rs})/K_s)$ along the extensions $K(\zeta_{p^rs}) / \mathbb{Q}_p(\zeta_{p^rs})/K_s$ and $K(\zeta_{p^rs}) / K/K_s$ we also obtain
\[
\operatorname{disc}(\Q_p(\zeta_{p^rs})/K_s)^{[K(\zeta_{p^rs}):\Q_p(\zeta_{p^rs})]}=
N_{ K/K_s}(\operatorname{disc}(K(\zeta_{p^rs})/K ))=\operatorname{disc}(K(\zeta_{p^rs})/K )^{[K:K_s]},
\]
so from \eqref{eq:disc} we get 
\[
 \operatorname{disc}(K(\zeta_{p^rs})/K )=\operatorname{disc}(\Q_p(\zeta_{p^rs})/\Q_p )^{\frac{1}{l_s}}=p^{p^{r-1}(pr-r-1)\frac{f_s}{l_s}}.
\]
This gives
\[
 \frac{\varphi(p^rs)}{[K(\zeta_{p^rs}):K]} v_p\left( \operatorname{disc}(K(\zeta_{p^rs})/K)\right)=
 \varphi(s)p^{r-1}(pr-r-1),
\]
and, by Equation \eqref{eq:thFormula},
\begin{align*}
 v_p(m(L/K))=&\frac{f_L}{2}\left(dp^n n-\sum_{r=1}^n\sum_{s\mid d} \varphi(s)p^{r-1}(pr-r-1)\right)\\
 =&\frac{f_L}{2}\left(dp^n n-d\sum_{r=1}^np^{r-1}(pr-r-1)\right)\\
 =&f_L d\frac{p^n-1}{p-1}.
\end{align*}
\end{proof}

There are also other known cases in which the associated order is maximal. From now on, we no longer assume that $p$ is odd.
\begin{lemma}[{\cite[Corollaire 3 to Théorème 1]{MR513880}}]\label{bergealmostmaximal}
Let $L/K$ be a totally ramified cyclic extension of $p$-adic fields with almost-maximal ramification.
Assume that $K$ is unramified over $\Q_p$. Then $\mathfrak{A}_{L/K}$ is a maximal order.
\end{lemma}
\begin{corollary}
\label{cor:almax}
Let $L/K$ be a Galois extension of $p$-adic fields with group $G$. Assume that $L/L^{nr}$ is cyclic of order $p^nd$ (with $(p,d)=1$) and almost-maximally ramified, and that $K/\mathbb{Q}_p$ is unramified.
\begin{enumerate}
\item $\mathcal{O}_L$ is free over $\mathfrak{A}_{L/L^{nr}}$, and  if $G$ is abelian it is also free over  $\mathfrak{A}_{L/K}$.
\item Assume $\mathcal{O}_L$ is free over $\mathfrak{A}_{L/K}$. Then $m(L/K)=p^{\frac{f_Ld(p^n-1)}{p-1}}$.
\end{enumerate}
\end{corollary}
\begin{proof}
\begin{enumerate}
\item 
$L/L^{nr}$ satisfies the assumptions of Lemma \ref{bergealmostmaximal}, hence $\mathfrak{A}_{L/L^{nr}}$ is a maximal order in $L^{nr}[G]$, and in particular $\mathcal{O}_L$ is free over $\mathfrak{A}_{L/L^{nr}}$ by Lemma \ref{associatedmaximal}. If $G$ is abelian, then 
 Proposition \ref{unramifiedindex} applies giving that $\mathcal{O}_L$ is free also over $\mathfrak{A}_{L/K}$.
\item 
By Proposition \ref{unramifiedindex} we see that $m(L/K)=m(L/L^{nr})=[\mathfrak{A}_{L/L^{nr}} : \mathcal{O}_{L^{nr}}[G_0]]$; since $\mathfrak{A}_{L/L^{nr}}$ is a maximal order of $L^{nr}[G_0]$, it suffices to compute the index of $\mathcal{O}_{L^{nr}}[G_0]$ inside a maximal order of $L^{nr}[G_0]$. Using Equation \eqref{formula-max-abeliano} and the same computation of discriminants as in the proof of Corollary \ref{cor:AbsolutelyAbelianCase} we get the result.
\end{enumerate}
\end{proof}
It is possible for the associated order to be maximal also when $K/\mathbb{Q}_p$ is ramified. This happens for instance whenever $L/K$ is an almost-maximally ramified Kummer extension of degree $p$: see Proposition \ref{prop:aeq0}, Corollary \ref{cor:AEqZeroOLIsAFree} and Remark \ref{rmk:greither}.
\begin{remark} If $G$ is abelian, a necessary condition for the associated order to be maximal is that $G$ is cyclic. For other considerations in the abelian case, see the survey of Thomas \cite{MR2760251}.
\end{remark}

\section{A local-global principle}
\label{sec:local-global}
Our main goal in this paper is to investigate the index $m(L/K)$ for a Galois extension $L/K$ of $p$-adic fields. In this section we shift slightly away from this setting, and show that in some special cases it is possible to study an analogous quantity $m(L/K)$ for extensions of number fields by reducing it to the local indices at the various completions. First of all note that, if $L/K$ is a Galois extension of number fields, the quantity
\[
 m(L/K):=\min_{\alpha \in \mathcal{O}_L} [\mathcal{O}_L : \mathcal{O}_K[G] \alpha]
\]
is still well-defined: in fact, as in the case of $p$-adic fields,
 any normal basis generator gives a finite index. However, in this case it is less interesting to define $m(L/K)$ as a minimum, since for instance it could be the case that the greatest common divisor of all the possible finite indices is strictly smaller than their minimum.
We can still make some interesting remarks about the globally defined index. As a preliminary observation, note that the associated order $\mathfrak{A}_{L/K}$ is defined exactly as for extensions of $p$-adic fields, and this continues to be the only possible candidate for an $\mathcal{O}_K$-order over which $\mathcal O_L$ is free. With the same proof as Proposition \ref{productindex}, we have the following.
\begin{proposition}\label{productindexglobal}
 Let $L/K$ be a Galois extension of number fields with Galois group $G$ such that $\mathcal{O}_L$ is free as an $\mathfrak{A}_{L/K}$-module. Then 
\[
m(L/K)=[\mathfrak{A}_{L/K}:\mathcal{O}_K[G]],
\]
and the minimal index is realised exactly by the generators of $\mathcal{O}_L$ as an $\mathfrak{A}_{L/K}$-module.   
\end{proposition}
Our main theorem is the following.
\begin{theorem}\label{thm:GlobalCase}
Let $L/\Q$ be a finite abelian extension. Then for every rational prime $p$ and every prime $\mathfrak{P}$ of $L$ above $p$ we have 
\[
\displaystyle
  v_p(m(L/\Q))=\frac{[L:\Q]}{e(\mathfrak{P}|p)f(\mathfrak{P}|p)}v_p(m(L_\mathfrak{P}/\Q_p)).
\]
 Hence
$
  m(L/\Q)=\prod_{\mathfrak{P}|p} m(L_\mathfrak{P}/\Q_p),
$
where the products runs over all the primes $\mathfrak{P}$ of $L$.
\end{theorem}
\begin{proof}
Let $G$ be the Galois group of $L/\Q$. By Proposition \ref{productindexglobal} and Leopoldt's Theorem \cite{MR0108479} we know that
$
m(L/\Q)=[\mathfrak{A}_{L/\Q}:\Z[G]].
$
Note that $\mathfrak{A}_{L/\Q}$ and $\Z[G]$ are two $\Z$-lattices in $\Q[G]$, hence both are free of rank $|G|$. Let $f:\Q[G]\rightarrow \Q[G]$ be a map of $\Q$-vector spaces such that $f(\mathfrak{A}_{L/\Q})=\Z[G]$. Then by Remark \ref{remark:nodvr} we have
\[
[\mathfrak{A}_{L/\Q}:\Z[G]]=|\det f|.
\]
Let $p$ be a rational prime; tensoring with $\Q_p$, the map $f$ induces a map $f_p:\Q_p[G]\rightarrow \Q_p[G]$ of $\Q_p$-vector spaces with the same determinant. Denoting by $\mathfrak{A}_{L/\Q,p}:=\mathfrak{A}_{L/\Q}\otimes_\Z\Z_p$, note that $f_p(\mathfrak{A}_{L/\Q,p})=\Z_p[G]$. By \eqref{eq:ind-det}, we deduce that 
$
 ([\mathfrak{A}_{L/\Q,p}:\Z_p[G]])=(\det f)
$
as ideals in $\Z_p$, which implies
\[
 v_p(m(L/\Q))=v_p([\mathfrak{A}_{L/\Q,p}:\Z_p[G]]).
\]
We now study the $\Z_p$-order $\mathfrak{A}_{L/\Q,p}$. Let $\mathfrak{P}$ be a prime of $L$ above $p$ and $D$ be its decomposition group in $G$. Note that $\mathfrak{A}_{L/\Q,p}$ can be seen as the associated order of 
\[
\mathcal{O}_{L,p}=\mathcal{O}_L\otimes_\Z\Z_p\cong \mathcal{O}_{L_\mathfrak{P}}\otimes_{\Z_p[D]}\Z_p[G]
\]
in $\Q_p[G]$, in the sense of Proposition \ref{associatedfree}. Then by \cite[Equation (2)]{MR747998} we have the relation
\[
 \mathfrak{A}_{L/\Q,p}=\bigcap_{g\in G}g\left(\mathfrak{A}_{L_\mathfrak{P}/\Q_p}\otimes_{\Z_p[D]}\Z_p[G]\right)g^{-1},
\]
and since $G$ is abelian this simply means
$
 \mathfrak{A}_{L/\Q,p}=\mathfrak{A}_{L_\mathfrak{P}/\Q_p}\otimes_{\Z_p[D]}\Z_p[G].
$
As $\Z_p[G]$ is free of rank $[G:D]=[L:\Q]/\left(e(\mathfrak{P}|p)f(\mathfrak{P}|p)\right)$ over $\Z_p[D]$, then we find that
\[
 [\mathfrak{A}_{L/\Q,p}:\Z_p[G]]=[\mathfrak{A}_{L_\mathfrak{P}/\Q_p}\otimes_{\Z_p[D]}\Z_p[G]:\Z_p[G]]=[\mathfrak{A}_{L_\mathfrak{P}/\Q_p}:\Z_p[D]]^{[G:D]}.
\]
The statement follows.
\end{proof}
From Corollary \ref{cor:AbsolutelyAbelianCase} we immediately deduce the following.
\begin{corollary}
 Let $L/\Q$ be a finite abelian extension and $p$ an odd prime. Let $p^nd$ be the ramification index of any prime of $L$ above $p$, with $(d,p)=1$. Then
 \[
  v_p(m(L/\Q))=\frac{[L:\Q](p^n-1)}{p^n(p-1)}.
 \]
\end{corollary}
\begin{remark}
 Let $\mathfrak{f}$ be the conductor of $L$. Then, since $\Q(\zeta_{\mathfrak{f}})/L$ is tamely ramified above odd primes (as follows for example from \cite[Theorem 8.2]{MR2078267}), in the notation of the above corollary we have that $v_p(\mathfrak{f})=n+1$ for all ramified primes $p$. In unpublished notes, Henri Johnston already established, by different methods, a similar formula relating $v_p(m(L/\Q))$ with the degree and conductor of an abelian extension.
\end{remark}

\section{Cyclic extensions of degree $p$}
\label{sec:degp}
\subsection{Preliminaries}
In this section we need the following notation. Let $L/K$ be a Galois extension of degree $p$ of a $p$-adic fields.  Let $G=\operatorname{Gal}(L/K)$, and let $\sigma$ be a generator of $G$. 
 We assume that $L/K$ is ramified with ramification jump $t$; we denote by $a \in \{0,\ldots,p-1\}$ the residue class modulo $p$ of $t$ and write $t=pt_0+a$. To simplify the notation, we will also denote simply by $e$ (instead of $e_K$) the absolute ramification index of $K$.
\begin{remark}\label{rmk:EGeqA}
Proposition \ref{prop:InequalitiesOneAndt} implies
$
\frac{p-1}p t=\frac{p-1}p a +(p-1)t_0\le e,
$
so we have the inequality
\[
e \geq \left\lceil \frac{p-1}p a\right\rceil +(p-1)t_0=a +(p-1)t_0
\]
since $p>a$.  Moreover, if $L/K$ is assumed \textit{not} to be almost-maximally ramified, then equality cannot hold and we have $e \geq a+(p-1)t_0+1$.
Recall that $L/K$ is maximally ramified if $t=\frac{ep}{p-1}$, and this corresponds to the case $a=0$.
\end{remark}
We will repeatedly need the following well-known lemma.
\begin{lemma}\label{lemma:DifferentValuations}
Suppose $L/K$ is totally ramified. Then any set of elements $\alpha_0, \ldots, \alpha_{p-1} \in \mathcal{O}_L$ such that $v_L(\alpha_i)=i$ forms a basis of $\mathcal{O}_L$ as a free $\mathcal{O}_K$-module.
Similarly, any set of elements $\beta_0,\ldots,\beta_{p-1} \in L$ such that $v_L(\beta_i) \equiv i \pmod{p}$ form a basis of $L$ as a $K$-vector space.
\end{lemma}
The case $a=0$ is special, and we handle it first. This case is well studied in the literature (see e.g.~\cite[Proposition 3]{MR296047}), but our approach seems to be new.
\begin{proposition}\label{prop:aeq0}
Suppose that $L/K$ is ramified and $a=0$. Then $L=K(\pi_K^{1/p})$ for a suitable uniformiser $\pi_K$ of $K$, we have $m(L/K)=p^{\frac{1}{2}[L:\mathbb{Q}_p]}$,  and $\omega = \sum_{i=0}^{p-1} c_i \pi_K^{i/p}$ achieves the minimal index if and only if every $c_i$ is a unit in $\mathcal{O}_K$.
\end{proposition}
\begin{proof}
By Proposition \ref{prop:InequalitiesOneAndt} we have $t=\frac{ep}{p-1}$, the extension $L/K$ is maximally ramified, $K$ contains the $p$-th roots of unity, and $L=K(\pi_K^{1/p})$ for a suitable choice of uniformiser $\pi_K$. Letting $\pi_L := \pi_K^{1/p}$, Lemma \ref{lemma:DifferentValuations} ensures that $1,\pi_L,\ldots,\pi_L^{p-1}$ is an $\mathcal{O}_K$-basis of $\mathcal{O}_L$. Let $\zeta_p = \sigma(\pi_L)/\pi_L$; it is a $p$-th root of unity.
Writing an arbitrary element of $\mathcal{O}_L$ as $\omega = \sum_{i=0}^{p-1} c_i \pi_L^i$ (with $c_i \in \mathcal{O}_K$), the Galois conjugates of $\omega$ are the elements 
$
\omega_j := \sigma^j(\omega)= \sum_{i=0}^{p-1} c_i \zeta_p^{ij} \pi_L^i
$,
so the index of $\mathcal{O}_K[G] \cdot \omega$ in $\mathcal{O}_L$ is the norm of the determinant of the matrix $M := \left( c_i \zeta_p^{ij} \right)_{i,j=0,\ldots,p-1}$. We have $\det(M) = \prod_{i=0}^{p-1} c_i \cdot \det \left(\zeta_p^{ij} \right)_{i,j=0,\ldots,p-1}$, and the latter is a Vandermonde determinant, so
\[
[\mathcal{O}_L : \mathcal{O}_K[G] \cdot \omega] = N(\det M) =  N\left( \prod_{i=0}^{p-1} c_i \right) \cdot \prod_{0 \leq i<j \leq p-1} N(\zeta_p^{j}-\zeta_p^i) =  N\left( \prod_{i=0}^{p-1} c_i \right) \cdot N(p \mathcal{O}_K)^{p/2},
\]
where the last equality follows from the well-known fact that $\operatorname{disc}(x^p-1)=\pm p^p$.
The minimum of $[\mathcal{O}_L : \mathcal{O}_K[G] \cdot \omega]$ is then achieved whenever $\prod_{i=0}^{p-1} c_i$ is a unit in $\mathcal{O}_K$, and is equal to $N(p\mathcal{O}_K)^{p/2}$. Notice that $\zeta_p \in K$ ensures that $p-1 \mid e$, hence $N(p\mathcal{O}_K)^{p/2}=p^{ef_Kp/2}$ is an integer. The statement follows from the fact that $ef_Kp = [K:\mathbb{Q}_p][L:K]=[L:\mathbb{Q}_p]$.
\end{proof}
\begin{remark}\label{rmk:KummerDegreepn}
An immediate extension of the argument in the previous proof shows that when $K$ contains the $p^r$-th roots of unity and $L=K(\sqrt[p^r]{\pi_K})$ we have 
\[
v_p(m(L/K))=\frac{1}{2} [K:\mathbb{Q}_p] v_p (\operatorname{disc}(x^{p^r}-1)) = \frac{1}{2} [K:\mathbb{Q}_p] rp^r = \frac{1}{2}r [L:\mathbb{Q}_p].
\]
\end{remark}
\begin{corollary}\label{cor:AEqZeroOLIsAFree}
Suppose that $L/K$ is ramified and $a=0$, so we can write $L=K(\pi_K^{1/p})$. Then $\mathcal{O}_L$ is free over $\mathfrak{A}_{L/K}$, and an $\mathcal{O}_K$-basis of the latter is given by $\pi_K^{-ei/(p-1)}(\sigma-1)^i$ for $i=0,\ldots,p-1$.  Moreover, $\omega=\sum_{i=0}^{p-1} c_i \pi_K^{i/p}$ generates $\mathcal{O}_L$ over $\mathfrak{A}_{L/K}$ if and only if every $c_i$ is a unit in $\mathcal{O}_K$.
\end{corollary}
\begin{proof}
Let $\pi_L=\pi_K^{1/p}$. We know that $1,\pi_L,\ldots,\pi_L^{p-1}$ is an $\mathcal{O}_K$-basis of $\mathcal{O}_L$. If $j\geq 1$, then we have $\pi_K^{-ei/(p-1)}(\sigma-1)^i \pi_L^j = \pi_K^{-ei/(p-1)} (\zeta_p-1)^{ij} \pi_L^j$, which has $L$-valuation $-e p i /(p-1) + \frac{ijep}{p-1}+j \geq 0$, so all the elements $\pi_K^{-ei/(p-1)}(\sigma-1)^i$ do in fact belong to $\mathfrak{A}_{L/K}$. Set $\mathfrak{B}=\bigoplus \mathcal{O}_K \pi_K^{-ei/(p-1)}(\sigma-1)^i$: it is a sub-$\mathcal{O}_K$-lattice of $\mathfrak{A}_{L/K}$ containing $\mathcal{O}_K[G]$.
The index $[\mathfrak{B}:\mathcal{O}_K[G]]$ is $N(\prod_{i=0}^{p-1} \pi_K^{ei/(p-1)}) = N( \pi_K^{\frac{1}{2} p e} ) = p^{\frac{1}{2}[L:\mathbb{Q}_p]}$, which by Proposition \ref{prop:aeq0} is also $m(L/K)$. Proposition \ref{productindex} then yields
\[
[\mathfrak{B}: \mathcal{O}_K[G]] = m(L/K) = [\mathfrak{A}_{L/K} : \mathfrak{B}] \cdot [\mathfrak{B}: \mathcal{O}_K[G]] \cdot \min_{\omega \in \mathcal{O}_L} [\mathfrak{A}_{L/K} \omega : \mathcal{O}_L],
\]
which implies that $\mathfrak{B}$ coincides with $\mathfrak{A}_{L/K}$ and that $\mathcal{O}_L$ is free over $\mathfrak{A}_{L/K}$. The last part of the statement follows from the analogous description of the elements that achieve the minimal index in Proposition \ref{prop:aeq0}.
\end{proof}
\begin{remark}\label{rmk:greither}
In the setting of Proposition~\ref{prop:aeq0} the associated order is maximal: this is well-known, but it can also be deduced from our results. In fact,  $m(L/K)$ is equal to the index of $\mathcal{O}_K[G]$ in the maximal order $\mathfrak{M}$ (see Equation \eqref{formula-max-abeliano}); on the other hand,  since $\mathcal{O}_L$ is free over the associated order by Corollary \ref{cor:AEqZeroOLIsAFree}, Proposition~\ref{productindex} shows that $m(L/K)=[\mathfrak{A}_{L/K}:\mathcal{O}_L]$, so $\mathfrak{A}_{L/K}=\mathfrak{M}$.

Proposition \ref{productindex} also shows that the elements realising the minimal index $m(L/K)$ are obtained as a fixed generator times a unit of $\mathfrak{A}_{L/K}=\mathfrak{M}$. More precisely we have $\mathfrak{M} \cong \mathcal{O}_K^p$, with the isomorphism given by a basis of orthogonal idempotents, and the action of $(u_0,\ldots,u_{p-1}) \in (\mathcal{O}_K^\times)^p$ on a generator $\alpha_0=c_0+c_1\pi_L+\cdots + c_{p-1}\pi_L^{p-1}$ sends it to $\sum_{i=0}^{p-1} u_i c_i \pi_L^i$, thus recovering in another way the description of all minimal generators given in Proposition \ref{prop:aeq0}.
Finally, $\omega_1$ and $\omega_2$ generate  the same  $\mathcal{O}_K[G]$-module if and only if $\omega_1/\omega_2$
is a unit of  $\mathcal{O}_K[G]$.  As a consequence,  the set of Galois modules $M_\omega: = \mathcal{O}_K[G] \omega$ realising the minimal index  is potentially very large, because $\mathfrak{M}^*$ is much bigger than $\mathcal{O}_K[G]^*$ in general: indeed, the quotient $\mathfrak{M}^*/\mathcal{O}_K[G]^*$ is finite, but (usually) nontrivial since the torsion subgroup of $\mathfrak{M}^*$ is isomorphic to $(\mathcal{O}_K^*)_{\operatorname{tors}}^p $ while $(\mathcal{O}_K[G])_{\operatorname{tors}}^*$ is isomorphic to $(\mathcal{O}_K^*)_{\operatorname{tors}}\times G$ (see \cite[Corollary 2.2]{MR2138721}).
\end{remark}
In all that follows it will be useful to work with a special element $f$ of the group ring, which we now define and whose properties we describe next.
\begin{definition}
We let $f := \sigma-1 \in \mathcal{O}_K[G]$.
\end{definition}
The following proposition summarises the key properties of the action of $f$.
\begin{proposition}\label{prop:PropertiesOff}
Suppose $L/K$ is totally ramified. The following hold:
\begin{enumerate}
\item the powers $1, f, f^2, \ldots, f^{p-1}$ of $f$ form an $\mathcal{O}_K$-basis of $\mathcal{O}_K[G]$;
\item the equality
$
f^p = - \sum_{j=1}^{p-1} {p \choose j} f^j
$
holds.
\end{enumerate}
Suppose in addition $a \neq 0$. Then:
\begin{enumerate}
\setcounter{enumi}{2}
\item $v_L(f^i \pi_L^a) = a+ it $ for $i=0, \ldots, p-1$;
\item the element $\pi_L^a$ generates a normal basis for $L/K$. Explicitly, $\{f^i \pi_L^a\}_{i=0,\ldots,p-1}$ is a $K$-basis of $L$;
\item we have $v_L(f^p  \pi_L^a) = ep+t+a$;
\item for every $x \in \mathcal{O}_L$ we have $v_L(f  x) \geq v_L(x) + t$.
\end{enumerate}
\end{proposition}
\begin{proof}
\begin{enumerate}[wide, labelwidth=!, labelindent=0pt]
\item By definition, an $\mathcal{O}_K$-basis of $\mathcal{O}_K[G]$ is given by the powers of $\sigma$. The claim then follows because the transition matrix between powers of $\sigma$ and powers of $f$ is unitriangular (that is, upper-triangular with all diagonal entries equal to $1$), hence invertible over $\mathcal{O}_K$.
\item By definition we have $\sigma^p=1$ in $\mathcal{O}_K[G]$, hence we obtain
\[
0 = \sigma^p-1=(1+f)^p-1 = \sum_{j=1}^{p-1} {p \choose j} f^j + f^p \Rightarrow f^p = - \sum_{j=1}^{p-1} {p \choose j} f^j.
\]
\item 
We begin by noticing that $t \geq 1$: indeed $t=-1$ would imply that $L/K$ is unramified, and $t=0$ would give $a=0$.
By \cite[IV.2, Proposition 5]{MR554237}, since $\sigma$ belongs to $G_t$ but not to $G_{t+1}$ we have 
\[
\frac{\sigma(\pi_L)}{\pi_L} = 1 + \pi_L^{t} u
\]
for some $u \in \mathcal{O}_L^\times$. Raising both sides to the $j$-th power, for any $j$ prime to $p$, gives
\[
\frac{\sigma(\pi_L^j)}{\pi_L^j} = 1 + \pi_L^{t} u_j
\]
where $u_j\in \mathcal{O}_L^\times$ since $(j,p)=1$. 
 Rearranging the previous equality gives
\[
\sigma(\pi_L^j)-\pi_L^j = \pi_L^{j+t} u_j,
\]
so we obtain $v_L(f \pi_L^j)=j+t$, provided that $(j,p)=1$. We now show the claim by induction on $0 \leq i \leq p-1$, the base case $i=0$ being trivial. We have
\[
f^{i+1}\pi_L^a = f \left( f^i \pi_L^a \right) = f\left(  \pi_L^{a+it} v_i \right)
\]
for some $v_i \in \mathcal{O}_L^\times$. Notice that $i<p-1$ (otherwise we would not need to prove anything for $i+1$), so $a+it \equiv a+ia \equiv a(i+1) \pmod{p}$ is nonzero. Using that $\sigma(v_i) \equiv v_i \pmod{\pi_L^{t+1}}$ (since $\sigma \in G_t$) we obtain $\sigma(v_i)=v_i+\gamma \pi_L^{t+1}$ for some $\gamma \in \mathcal{O}_L$, so we can write
\[
\begin{aligned}
f(\pi_L^{a+it} v_i) & = \sigma( \pi_L^{a+it} ) \sigma( v_i ) - \pi_L^{a+it} v_i \\
& = (\pi_L^{a+it} + \pi_L^{a+(i+1)t} u_{a+it} ) (v_i + \pi_L^{t+1} \gamma) - \pi_L^{a+it} v_i \\
& =  \pi_L^{a+(i+1)t} u_{a+it}v_i +\pi_L^{a+(i+1)t+1} \gamma + \pi_L^{a+(i+2)t+1} u_{a+it} \gamma
\end{aligned}
\]
Since $ u_{a+it}$ and $ v_i$ are units,
last expression contains only one summand of minimal valuation, namely $\pi_L^{a+(i+1)t} u_{i+1}u_i$, so the valuation of $f^{i+1}(\pi_L^a)$ is $a+(i+1)t$ as desired.

\item By part (3), the $L$-valuations of the elements $\{f^i \pi_L^a\}_{i=0,\ldots,p-1}$ are all distinct modulo $p$. The claim then follows from Lemma \ref{lemma:DifferentValuations}.

\item Follows from (2) and (3) upon noticing that $v_L\left( {p \choose j} \right)=v_L(p)=ep$ for all $j=1,\ldots,p-1$.

\item By (4) we may write every element of $\mathcal{O}_L$ as $\sum_{i=0}^{p-1} c_i f^i \pi_L^a$ with $c_i \in K$.  The claim then follows combining (3), (5) and Proposition \ref{prop:InequalitiesOneAndt}: together they show that we have $v_L (f(c_i f^i \pi_L^a)) \geq v_L (c_i f^i \pi_L^a)+t$ for all $i$.
\end{enumerate}
\end{proof}
The following quantity will be important in all that follows.
\begin{definition}\label{def:nui}
We set
$\nu_i = \left\lfloor \frac{it+a}{p} \right\rfloor$.
\end{definition}

In terms of this notation, we have very explicit descriptions of the ring of integers of $L$. %
\begin{theorem}\label{thm:StructureOfStuff}
Assume $a \neq 0$. The set $\{\pi_K^{-\nu_i} f^i \pi_L^a\}_{i=0,\ldots,p-1}$ is a $\mathcal{O}_K$-basis of $\mathcal{O}_L$.
\end{theorem}
\begin{proof}
The $L$-valuation of $\pi_K^{-\nu_i} f^i\pi_L^a$ is $v_L(f^i\pi_L^a)-\nu_i p$, which by Proposition \ref{prop:PropertiesOff} (3) and the definition of $\nu_i$ is equal to $a+it - p \left\lfloor \frac{it+a}{p} \right\rfloor$. This is nothing but the remainder in the division of $a+it$ by $p$, so it is a non-negative integer in the interval $[0,p-1]$. Furthermore, it is congruent modulo $p$ to $(i+1)a$, so (since $a \not\equiv 0 \pmod{p}$) the $L$-valuations of the $p$ elements $\pi_K^{-\nu_i} f^i\pi_L^a$ for $i=0,\ldots,p-1$ are precisely $\{0,\ldots,p-1\}$ in some order. The claim follows  from Lemma \ref{lemma:DifferentValuations}.
\end{proof}
We conclude this section of preliminaries with a technical lemma that we will need several times.
\begin{lemma}\label{lemma:InequalityOne}
The inequality $e \geq \nu_{p-1} \geq \nu_{s}$ holds for all $s=0,\ldots,p-1$. Moreover, $e > \nu_s$ holds for all $s=0,\ldots,p-2$.
\end{lemma}
\begin{proof}
The sequence $s \mapsto \nu_s$ is increasing in $s$, so it suffices to prove $e \geq \nu_{p-1}$ and $e > \nu_{p-2}$.
By Remark \ref{rmk:EGeqA} we have $e \geq (p-1)t_0 + a$. Now simply observe that $\nu_{p-1} = \left\lfloor \frac{a+(p-1)t}{p} \right\rfloor = \left\lfloor \frac{a+(p-1)pt_0 + (p-1)a}{p} \right\rfloor = a+(p-1)t_0 \leq e$, and that
\[
\nu_{p-2} = \left\lfloor \frac{a+(p-2)t}{p} \right\rfloor \leq \frac{a+(p-2)t}{p}=(p-2)t_0 + \frac{(p-1)a}{p} < a+(p-1)t_0.
\]
\end{proof}

\subsection{The minimal index for cyclic extensions of degree $p$}
Our purpose in this section is to prove the following result, which -- combined with Proposition \ref{prop:aeq0} -- will give a proof of Theorem \ref{thm:IntroCyclicDegreep}.
\begin{theorem}\label{thm:FormulaCyclicExtensionsDegreeP}
Suppose $a \neq 0$. Let $\nu_i$ be as in Definition \ref{def:nui} and let
$\displaystyle 
\mu := \min_{0\leq i\leq p-1}(i e-(p-1)\nu_i).
$
Then $m(L/K)=p^{f_K(\mu + \sum_{i=0}^{p-1} \nu_i)}$.
\end{theorem}
We identify the elements of $L$ with the vectors of their coordinates with respect to the $K$-basis of $L$ given by $\{f^i \pi_L^a\}_{i=0,\ldots,p-1}$. 

We let  $v:=f^p \pi_L^a=  - \sum_{i=1}^{p-1} {p \choose i} f^i \pi_L^a$ and, with the identification above, we have $v=\begin{pmatrix}
0 \\ -{p \choose 1} \\ -{p \choose 2} \\ \vdots \\ -{p \choose p-1}
\end{pmatrix}$. Also denote by $\lambda_i : L \to K$ the dual basis of $\{f^i \pi_L^a\}_{i=0,\ldots,p-1}$, so that every $\beta \in L$ can be written as $\beta = \sum_{i=0}^{p-1} \lambda_i(\beta) f^i \pi_L^a$.
\begin{lemma}\label{lemma:fiv}
The following hold for all $j =0,\ldots,p-1$:
\begin{enumerate}
\item $\lambda_0 (f^jv)=0$;
\item $v_K(\lambda_i (f^jv)) \geq 2e$ for $i=1,\ldots,j$;
\item $v_K(\lambda_i (f^jv))=e$ for $i=j+1,\ldots,p-1$.
\end{enumerate} 
\end{lemma}
\begin{proof}
For $j=0$ the statement follows from the well-known fact that $v_p \left({p \choose i} \right)=1$ for all $i=1,\ldots,p-1$.
The general case then follows by an immediate induction.
\end{proof}

We now compare several $\mathcal{O}_K$-lattices in $L$, namely $\mathcal{O}_L$, $\mathcal{O}_K[G]   \pi_L^a$, and the lattice $\mathcal{O}_K[G]   \beta$ generated by a normal basis generator $\beta \in \mathcal{O}_L$. To each of these lattices $\Lambda$ we associate a matrix whose $i$-th column is the vector of coordinates of the $i$-th generator of $\Lambda$ in the basis $\{f^i\pi_L^a\}_{i=0,\ldots,p-1}$. The theory of Section \ref{sect:Preliminaries} will then allow us to describe the index $[\mathcal{O}_L : \mathcal{O}_K[G]   \beta]$ in terms of the determinants of these matrices.
By Theorem \ref{thm:StructureOfStuff} we have
$
 \mathcal{O}_L=\langle 1,\pi_K^{-\nu_1}f,...,\pi_K^{-\nu_{p-1}}f^{p-1}\rangle_{\mathcal{O}_K}\pi_L^a,
$
so the matrix of the lattice $\mathcal{O}_L$ is simply
\begin{equation}\label{eq:TransitionMatrixOL}
B := \begin{pmatrix}
1 & 0 & 0 &  \cdots & 0 \\
0 & \pi_K^{-\nu_1} & 0 & \cdots & 0  \\
0 & 0 & \pi_K^{-\nu_2} & \cdots & 0 \\
\vdots & \vdots & \vdots &  \cdots & \vdots \\
0 & 0 & 0 & \cdots & \pi_K^{-\nu_{p-1}}
\end{pmatrix},
\end{equation}
while the lattice $\mathcal{O}_K[G]\pi_L^a=\langle 1,f,...,f^{p-1}\rangle_{\mathcal{O}_K}\pi_L^a$ corresponds to the $p \times p$ identity matrix.

Now let $\beta\in \mathcal{O}_L$ be a normal basis generator of the extension $L/K$. To write the matrix associated  to the lattice $\langle 1,f,...,f^{p-1}\rangle_{\mathcal{O}_K}\beta$ we need to compute the coordinates of $f^j \beta$ in the basis $\{f^i \pi_L^a\}_{i=0,\ldots,p-1}$.
Writing
$
\beta = \begin{pmatrix}
c_0 \\ c_1 \\ \vdots \\ c_{p-1}
\end{pmatrix}
$
we have
$
f\beta = \begin{pmatrix}
0 \\ c_0 \\ c_1 \\ \vdots \\ c_{p-2}
\end{pmatrix} + c_{p-1} v,
$
and iterating $f$ we find
\begin{equation}\label{coefficients}
f^j \beta = \sum_{k=0}^{p-1-j} c_k f^{k+j} \pi_L^a + \sum_{k=0}^{j-1} c_{p-1-k} f^{j-1-k} v.
\end{equation}
Our goal is to estimate the determinant of the matrix corresponding to the lattice $\mathcal{O}_K[G]\beta$. To do this we estimate the valuation of the coefficients of each $f^j\beta$, which are the entries of the aforementioned matrix, looking at the two separate terms of \eqref{coefficients}. We start with the following observation.
\begin{remark}\label{rmk:LowerBoundValuationsCi}
Since $\mathcal{O}_L$ is free over $\mathcal{O}_K$ with basis $\{\pi_K^{-\nu_i} f^i \pi_L^a\}_{i=0,\ldots,p-1}$ we have $v_K(c_i) \geq -\nu_i$ for all $i=0,\ldots,p-1$.
\end{remark}
\begin{lemma}\label{lemma:InequalityOnValuations1}
The following inequality holds for all $j=1,\ldots,p-1$ and $i=1,\ldots,j$:
\[
v_K \left( \lambda_i\left( \sum_{k=0}^{j-1} c_{p-1-k} f^{j-1-k} v \right) \right) \geq e - \nu_{p-1-j+i}.
\]
\end{lemma}
\begin{proof}
As $\lambda_i$ is linear, the quantity we want to estimate is at least
\[
\min_{k=0,\ldots,j-1} v_K( \lambda_i( c_{p-1-k} f^{j-1-k} v ) )=
\min_{k=0,\ldots,j-1} \left( v_K(c_{p-1-k}) + v_K( \lambda_i(  f^{j-1-k}v ) ) \right)
\]
Now if $i \leq j-1-k$ we have $v_K(\lambda_i(f^{j-1-k} v)) \geq 2e$ by Lemma \ref{lemma:fiv}, and by the same lemma for all $i$ we have $v_K(\lambda_i(f^{j-1-k} v)) \geq e$,
so the previous expression is bounded below by
\[
\min\left\{ \min_{k=0,\ldots,j-1-i} (2e + v_K(c_{p-1-k}) ), \min_{k=j-i,\ldots,j-1} (e + v_K(c_{p-1-k}) )  \right\}.
\]
Now recall from Remark \ref{rmk:LowerBoundValuationsCi} that $v_K(c_{p-1-k}) \geq -\nu_{p-1-k}$, hence we get the lower bound
\[
\min \left\{ 
\min_{k=0,\ldots,j-1-i} (2e -\nu_{p-1-k} ), \min_{k=j-i,\ldots,j-1} (e - \nu_{p-1-k} )\right\}.
\]
Since the sequence $\nu_s$ is increasing in $s$, each of the two internal minima is achieved for $p-1-k$ as large as possible, that is, for the minimal admissible value of $k$ in each case. Thus we are reduced to studying
\[
\min \{ 2e -\nu_{p-1}, e - \nu_{p-1-j+i} \};
\]
Lemma \ref{lemma:InequalityOne} implies $2e -\nu_{p-1} \geq e \geq e - \nu_{p-1-j+i}$, so the minimum above is in fact equal to $e - \nu_{p-1-j+i}$, and we are done.
\end{proof}

The following is a variant of the previous lemma for $i=j+1,\ldots,p-1$.
\begin{lemma}\label{lemma:InequalityOnValuations2}
The following inequality holds for all $j=1,\ldots,p-1$ and $i=j+1,\ldots,p-1$:
\[
v_K \left( \lambda_i\left( \sum_{k=0}^{j-1} c_{p-1-k} f^{j-1-k} v \right) \right) \geq e-\nu_{p-1}
\]
\end{lemma}
\begin{proof}
This is proved as the previous lemma, but is in fact easier, because for every index $j$ we have $v_K(c_{p-1-k}) \geq -\nu_{p-1-k} \geq -\nu_{p-1}$ and $v_K(\lambda_i (f^{j-1-k} v)) \geq v_K(p) = e$.
\end{proof}
Combining Equation \eqref{coefficients} (which implies in particular $\lambda_0(f^j \beta)=0$ for $j > 0$) with Lemmas \ref{lemma:InequalityOne}, \ref{lemma:InequalityOnValuations1}, and \ref{lemma:InequalityOnValuations2} allows us to conclude that the coefficients $a_{ij} := \lambda_i(f^j \beta)$ have valuation at least
\[
v_K(a_{ij}) \geq \begin{cases}
- \nu_{i-j} \text{ for } i \geq j \\
e - \nu_{p-1-j+i} \text{ for } i < j
\end{cases}
\]
Notice moreover that $a_{0,j}=0$ for $j=1,\ldots,p-1$. 
Let $A=(a_{ij})_{i,j=0,\ldots,p-1}$ be the matrix corresponding to the lattice $\mathcal{O}_K[G] \beta$ in the basis $f^i\pi_L^a$. 
Recalling that the subgroup index is the norm of the index as a module and relation \eqref{eq:ind-det}, we get
\begin{equation}\label{eq:DeterminantsToIndices}
\begin{split}
[\mathcal{O}_L : \mathcal{O}_K[G]  \beta] &= [\mathcal{O}_L : \mathcal{O}_K[G] \pi_L^a] \, [\mathcal{O}_K[G] \pi_L^a:\mathcal{O}_K[G]  \beta] \\
&=N\left( \det(B^{-1}) \right)N\left(\det(A) \right)=N( \pi_K^{\sum_{i=0}^{p-1} \nu_i} ) N(\det(A)).
\end{split}
\end{equation}
It follows that the index $[\mathcal{O}_L : \mathcal{O}_K[G] \beta]$ is minimal if and only if the valuation of $\det(A)$ is minimal.
We are now in a position to prove our exact formula for $m(L/K)$ in the case of cyclic extensions of order $p$.
Let
\[
\mu := \min_{0\leq i\leq p-1}(i e-(p-1)\nu_i).
\]
We will prove that $v_K(\det(A)) \geq \mu$, and that for suitable $\beta$ equality holds.

In reading the next few paragraphs, the reader may find it useful to bear in mind the following description of $A$, which is equivalent to the inequalities discussed above: we have
\begin{equation}\label{eq:AFleshAndBones}
\begin{split}
 A&=\begin{pmatrix}
d_0 & 0 & 0 &  \cdots & 0 \\
\pi_K^{-\nu_1}d_1 & d_0 & 0 & \cdots & 0  \\
\pi_K^{-\nu_2}d_2 & \pi_K^{-\nu_1}d_1 & d_0 & \cdots & 0 \\
\vdots & \vdots & \vdots & \vdots &  \vdots \\
\pi_K^{-\nu_{p-1}}d_{p-1} & \pi_K^{-\nu_{p-2}}d_{p-2} & \pi_K^{-\nu_{p-3}}d_{p-3} & \cdots & d_0
\end{pmatrix}\\
&+\begin{pmatrix}
0 & 0 & 0 & 0 &  \cdots & 0 \\
0 & \pi_K^{e-\nu_{p-1}} \gamma_{1,1} & \pi_K^{e-\nu_{p-2}} \gamma_{1,2} & \pi_K^{e-\nu_{p-3}} \gamma_{1,3} & \cdots & \pi_K^{e-\nu_{1}} \gamma_{1,p-1} \\
0 & \pi_K^{e-\nu_{p-1}} \gamma_{2,1} & \pi_K^{e-\nu_{p-1}} \gamma_{2,2} & \pi_K^{e-\nu_{p-2}} \gamma_{2,3} & \cdots & \pi_K^{e-\nu_{2}} \gamma_{2,p-1} \\
0 & \pi_K^{e-\nu_{p-1}} \gamma_{3,1} & \pi_K^{e-\nu_{p-1}} \gamma_{3,2} & \pi_K^{e-\nu_{p-1}} \gamma_{3,3} & \cdots & \pi_K^{e-\nu_{3}} \gamma_{3,p-1} \\
\vdots & \vdots & \vdots& \vdots &  \cdots & \vdots \\
0 & \pi_K^{e-\nu_{p-1}} \gamma_{p-1,1} & \pi_K^{e-\nu_{p-1}} \gamma_{p-1,2} & \pi_K^{e-\nu_{p-1}} \gamma_{p-1,3} & \cdots & \pi_K^{e-\nu_{p-1}} \gamma_{p-1,p-1}
\end{pmatrix}
\end{split}
\end{equation}
where we have written $c_i=d_i \pi_K^{-\nu_i}$ with $d_i \in \mathcal{O}_L$ and where every $\gamma_{i,j}$ is in $\mathcal{O}_L$. The two terms of $A$ correspond to the two terms of \eqref{coefficients}.

Let $\tilde{A}$ denote the $(p-1) \times (p-1)$ submatrix $(a_{i,j})_{i,j = 1,\ldots,p-1}$, that is, the bottom-right block of $A$. Since the only non-zero coefficient in the first row of $A$ is $a_{0,0}=c_0=d_0$, the Laplace expansion of $\det A$ gives
$\det(A) = d_0 \det(\tilde{A})$. Define now
\[
b_{i,j} = \begin{cases}
\pi_K^{e - \nu_{(p-1)+(i-j)}}, \text{ if } i<j \\
\pi_K^{-\nu_{i-j}}, \text{ if } i \geq j
\end{cases}
\]
for $i,j = 1,\ldots,p-1$: by what we have already shown, $v_K(a_{i,j}) \geq v_K(b_{i,j})$ for all $i,j=1,\ldots,p-1$. Consider the Leibniz formula for $\det(\tilde{A})$:
\[
\det(\tilde{A}) = \sum_{\sigma \in S_{p-1}} \operatorname{sgn}(\sigma) \prod_{i=1}^{p-1} a_{i, \sigma(i)}.
\]
We will show that $\det(\tilde{A})$ (hence also $\det(A)=d_0 \det(\tilde{A})$) has valuation at least $\mu$ by proving that $v_K\left( \prod_{i=1}^{p-1} a_{i, \sigma(i)} \right) \geq \mu$ for all $\sigma$. Since $v_K(a_{i,\sigma(i)}) \geq v_K(b_{i,\sigma(i)})$, it is enough to prove that the inequality $v_K\left( \prod_{i=1}^{p-1} b_{i, \sigma(i)} \right) \geq \mu$ holds for every $\sigma$.
Given the definition of the coefficients $b_{i,j}$ we have that in general
\[
 \prod_{i=1}^{p-1} b_{i,\sigma(i)}=\pi_K^{ke-\sum_{j\in S}\nu_j},
\]
where $k=k(\sigma)$ is a non-negative integer and $S=S(\sigma)$ is a multiset of indices of cardinality $p-1$ (taking into account the multiplicities).
\begin{lemma}\label{lemma:RelationKS}
If for some $\sigma$ we have $ \prod_{i=1}^{p-1} b_{i,\sigma(i)}=\pi_K^{ke-\sum_{j\in S}\nu_j}$, then $\sum_{j\in S}j=(p-1)k$.
\end{lemma}
\begin{proof}
By the definition of $b_{i,j}$ we have
\[
\begin{aligned}
v_K \left( \prod_{i} b_{i,\sigma(i)} \right) & = \sum_i \begin{cases}
e - \nu_{(p-1)+(i-\sigma(i))}, \text{ if } i<\sigma(i) \\
-\nu_{i - \sigma(i)}, \text{ if } i \geq \sigma(i)
\end{cases} \\
& = k e - \sum_{j \in S} \nu_j,
\end{aligned}
\]
for $k=\#\{ i \in \{1,\ldots,p-1\} : i < \sigma(i) \}$ and $S$ the multiset 
$
\left\{s_i \bigm\vert i \in \{1,\ldots,p-1\} \right\},
$ 
where 
\[
s_i= \begin{cases} (p-1)+(i-\sigma(i))  \text{ if } i<\sigma(i) \\
i-\sigma(i)  \text{ if } i \geq \sigma(i).
 \end{cases}
\]
In particular,
\[
 \begin{aligned}
 \sum_{s \in S} s & = \sum_{i=1}^{p-1} s_i \\
 & = \sum_{i=1}^{p-1} \begin{cases} (p-1)+(i-\sigma(i))  \text{ if } i<\sigma(i) \\
i-\sigma(i)  \text{ if } i \geq \sigma(i)  \end{cases}
\\
& =k(p-1) + \sum_{i=1}^{p-1} \begin{cases} i-\sigma(i)  \text{ if } i<\sigma(i) \\
i-\sigma(i)  \text{ if } i \geq \sigma(i)
 \end{cases} \\
 & = k(p-1) + \sum_{i=1}^{p-1} i - \sum_{i=1}^{p-1} \sigma(i) \\
 & = k(p-1).
\end{aligned}
\]
\end{proof}
Now consider the function $g:\{0,...,p-1\} \rightarrow \mathbb{Q}$ given by $j \mapsto \frac{j}{p-1}e-\nu_j$
and let $i_m$ be an index that realises the minimum of $g$. Notice that $(p-1)g(i_m)$ realises the minimum of $j \mapsto je - (p-1)\nu_j$, that is, $(p-1)g(i_m)=\mu$.
Hence for every $\sigma\in S_p$ we have
\[
\begin{aligned}
 v_K\left(\prod_ib_{i,\sigma(i)}\right) & 
 = ke-\sum_{j\in S}\nu_j 
 \\ & =\frac{\sum_{j\in S}j}{p-1}e-\sum_{j\in S}\nu_j  & (\text{Lemma \ref{lemma:RelationKS}})
 \\ & =\sum_{j\in S}g(j)\geq (p-1)g(i_m)=\mu
\end{aligned}
\]
as claimed. 
Finally, we show that for suitable $\beta\in \mathcal{O}_L$ we do in fact have $v_K(\det A)=\mu$. If the minimum of the function $g(j)$ is $0$ we can take $\beta=\pi_L^a$ (in this case $\mu=0$ and the matrix $A$ is simply the identity), otherwise
we claim that $\beta=\pi_L^a+\pi_K^{-\nu_{i_m}}f^{i_m}\pi_L^a$ does the job (in terms of Equation \eqref{eq:AFleshAndBones} we are choosing $d_0=d_{i_m}=1$ and $d_i =0 $ for $i \neq 0, i_m$). Notice that we may assume $i_m>0$, for otherwise the minimum would be $0$. Similarly, we may assume $i_m < p-1$, because $g(p-1) \geq 0$ by Lemma \ref{lemma:InequalityOne}.
To show the claim, we begin by noticing that specialising Equation \eqref{coefficients} to our choice of $\beta$ we obtain
\begin{equation}\label{eq:ValuationsSpecialA}
f^k \beta = f^k \pi_L^a + \begin{cases} \pi_K^{-\nu_{i_m}} f^{k+i_m} \pi_L^a, \text{ if } k+i_m \leq p-1 \\
\pi_K^{-\nu_{i_m}} f^{k+i_m-p} v, \text{ if } k +i_m \geq p,
\end{cases}
\end{equation}
We now show that in the Leibniz development of $\det(A)$ there is precisely one summand of valuation $\mu$, while all other have strictly larger valuation, whence $v_K(\det A)=\mu$ will follow. Observe that with our choice we have $d_0=1$, so $\det(A)=\det(\tilde{A})$; we will work with this latter determinant and its Leibniz expansion.
Let $\tau : \{1,\ldots,p-1\} \to \{1,\ldots,p-1\}$ be defined by
\[
\tau(j) = \begin{cases}
\begin{array}{ll}
i_m+j & \text{ if } j+i_m \leq p-1 \\
i_m+j+1-p & \text{ if } j+i_m \geq p \\
\end{array}
\end{cases}
\]
One checks that the function $\tau$ is a permutation (indeed, it's easily seen to be injective). The Leibniz term
$
\operatorname{sgn} \tau \cdot \prod_{j=1}^{p-1} a_{\tau(j),j}
$
in the development of $\det(\tilde{A})$ has valuation
\[
\sum_{j=1}^{p-1} v_K(a_{\tau(j),j}) = \sum_{j=1}^{p-1-i_m} (-\nu_{i_m}) + \sum_{j=p-i_m}^{p-1} (e-\nu_{i_m})=(p-1)g(i_m)=\mu.
\]
We now show that all other terms in the Leibniz development have strictly greater valuation. Since we already know that $v_K \left( \prod_{i=1}^{p-1} a_{i,\sigma(i)} \right) \geq \mu$, it suffices to prove the following.
\begin{proposition}\label{prop:SingleLeibnizTermWithMinimalValuation}
Let $\sigma$ be a permutation such that $v_K \left( \prod_{i=1}^{p-1} a_{i,\sigma(i)} \right) = \mu$. Then $\sigma = \tau^{-1}$.
\end{proposition}
By Equation \eqref{eq:ValuationsSpecialA} and Lemma \ref{lemma:fiv} we know that %
\begin{equation}\label{eq:ExplicitValuations}
v_K(a_{i, j}) = \begin{cases}
0, \text{ if } i=j \\
\infty, \text{ if } j \leq p-i_m-1 \text{ and }i \neq j \text{ and } i \neq j+i_m \\
-\nu_{i_m}, \text{ if } j \leq p-i_m-1 \text{ and } i=j+i_m \\
e- \nu_{i_m}, \text{ if } j \geq p-{i_m} \text{ and } i \neq j \text{ and } i \geq j+i_m-(p-1) \\
v_K(a_{i, j}) \geq 2e-\nu_{i_m}, \text{ if } j \geq p-i_m \text{ and } i \leq j+i_m-p
\end{cases}
\end{equation}
We set for simplicity $\nu := \nu_{i_m}$. We shall need the following observation multiple times:
\begin{lemma}
We have $e-\nu > 0$.
\end{lemma}
\begin{proof}
Follows from Lemma \ref{lemma:InequalityOne} and the fact that, as already observed, $i_m$ cannot be $p-1$.
\end{proof}
By what we have already shown, the inequalities
\[
v_K\left( a_{i, \sigma(i)} \right) \geq v_K\left( b_{i, \sigma(i)} \right), \quad \sum_{i=1}^{p-1} v_K\left( b_{i, \sigma(i)} \right) \geq \mu
\]
hold for every permutation $\sigma$. If $\sigma$ is such that $v_K\left(\prod_{i=1}^{p-1} a_{i, \sigma(i)} \right)=\mu$ holds, then equality must hold in both inequalities above. In particular, we must have $v_K(a_{i,\sigma(i)}) = v_K(b_{i,\sigma(i)})$ for all $i$. We now remark that $v_K(b_{i,j})$ is either of the form $-\nu_{i-j} \leq 0 < 2e-\nu$ or of the form $e-\nu_{(p-1)+(i-j)} < 2e-\nu$. It follows that none of the terms $a_{i, \sigma(i)}$ can have valuation $2e-\nu$ (or more), for otherwise its valuation would be strictly greater than the valuation of the corresponding  term $b_{i,\sigma(i)}$. Hence, by \eqref{eq:ExplicitValuations}, for each $i$ we have $v_L(a_{i,\sigma(i)})=v_L(b_{i,\sigma(i)}) \in \{0, -\nu, e-\nu\}$. Next suppose that $v_K(a_{i,\sigma(i)}) = e-\nu > 0$. Then also $v_K(b_{i,\sigma(i)}) > 0$, so $b_{i,\sigma(i)}$ must be of the form $\pi_K^{e-\nu_{(p-1) + (i-\sigma(i))}}$, and in order for equality to hold we must have $\nu_{(p-1) + (i-\sigma(i))} = \nu$. Finally, if $a_{i,\sigma(i)}$ has valuation $0=\nu_0$ or $-\nu$, then so does the corresponding $b_{i,\sigma(i)}$. It follows that
\[
v_K\left( \prod_{i=1}^{p-1} b_{i,\sigma(i)} \right) = ke - \sum_{j \in S} \nu_j
\]
for some multiset $S$ with the property that every $\nu_j$ is either equal to $\nu$ or to $0$, and (by Lemma \ref{lemma:RelationKS}) $(p-1)k = \sum_{j \in S} j$. Notice that we are not claiming that every $j$ is either $0$ or $i_m$, but just that $\nu_j \in \{0, \nu\}$ for every $j \in S$. We may then continue our chain of inequalities as follows:
\[
\begin{aligned}
\mu & = v_K \left( \prod_{i=1}^{p-1} a_{i,\sigma(i)} \right) 
= v_K \left( \prod_{i=1}^{p-1} b_{i,\sigma(i)} \right)\\
& = ke - \sum_{j \in S} \nu_j = \sum_{j \in S} \left( \frac{j}{p-1}e - \nu_j \right) \\
& = \sum_{j \in S} g(j) \geq (p-1) g(i_m) = \mu,
\end{aligned}
\]
which again yields that equality must hold everywhere. In particular, we must have that for all $j \in S$ the equality $g(j)=g(i_m)$ holds, so that for no $j \in S$ we can have $\nu_j=0$. By the above remarks, this implies that no valuation $v_K(b_{i,\sigma(i)})$ is equal to $0$, hence no valuation $v_K(a_{i,\sigma(i)})$ is equal to $0$ either. Since clearly no $v_K(a_{i,\sigma(i)})$ can be equal to $\infty$ (for otherwise $\prod_{i=1}^{p-1} a_{i,\sigma(i)}=0$ does not have valuation $\mu$), by Equation \eqref{eq:ExplicitValuations} we have shown that $v_K(a_{i,\sigma(i)})$ is either $-\nu$ or $e-\nu$ for every $i=1,\ldots,p-1$. 

Let us recapitulate which terms of $\tilde{A}$ are such that $v_K(a_{i,j})=v_K(b_{i,j})\in \{-\nu,e-\nu\}$ (these are the only entries $a_{i,\sigma(i)}$ that may be involved in a Leibniz term of minimal valuation): from \eqref{eq:ExplicitValuations} we see that if $1\leq j\leq p-1-i_m$ then we can only have $v_K(a_{\sigma^{-1}(j),j})=-\nu$ and $\sigma^{-1}(j)=j+i_m$; if $p-i_m\leq j \leq p-1$ then $v_K(a_{\sigma^{-1}(j),j})=e-\nu$ and $\sigma^{-1}(j)\geq j-p+1+i_m$ (more precisely, even if this is not necessary for the proof, note that we also have $\sigma^{-1}(j)\leq  j-p+1+\max\{i\geq i_m:\nu_i=\nu_{i_m}\}$). While reading the rest of the proof, the reader may find it useful to look at the following matrix, where in correspondence of such terms $(i,j)$ we display their valuation $v_K(a_{i,j})$:
\[
\begin{matrix}
\phantom{a} \\ \phantom{\vdots} \\ \phantom{\vdots} \\ \phantom{c} \\ \phantom{\vdots} (i_m+1)\text{-th row} \rightarrow \\ \phantom{e} \\ \phantom{\ddots} \\ \phantom{h}
 \end{matrix}
\begin{pmatrix}
 &  &  &  & e-\nu & & & \\
 &  &  &  & \vdots & e-\nu & & \\
 &  &  &  & e-\nu & \vdots & \ddots & \\
 &  &  &  & & e-\nu & & e-\nu \\
-\nu &  &  &  & & & \ddots & \vdots \\
 & -\nu  &  &  & & & & e-\nu \\
 &  & \ddots &  & & & & \\
 &  &  & -\nu & & & & \\
\end{pmatrix}\begin{matrix}
\phantom{a} \\ \phantom{\vdots} \\ \phantom{\vdots} \\ \leftarrow i_m\text{-th row} \\ \phantom{\ddots}  \\ \phantom{e} \\ \phantom{\ddots} \\ \phantom{h}
 \end{matrix}
\]
\begin{lemma}\label{lemma:OnlyPermutation}
Let $\sigma : \{1,\ldots,p-1\} \to \{1,\ldots,p-1\}$ be a permutation such that $v_K(a_{i,\sigma(i)})\in \{-\nu, e-\nu\}$ for all $i=1,\ldots,p-1$. Then $\sigma=\tau^{-1}$.
\end{lemma}
\begin{proof}
For every $i=1,\ldots,p-1$ the element $(i,\sigma(i))$ must be among those displayed in the matrix above (call such elements \textit{admissible}). For $j \leq p-i_m-1$, the $j$-th column contains only one admissible element, in position $j+i_m$, hence $\sigma^{-1}(j)=j+i_m$ for $j=1,\ldots,p-1-i_m$, that is, $\sigma(i)=i-i_m=\tau^{-1}(i)$ for $i=i_m+1,\ldots,p-1$. It remains to determine $\sigma(i)$ for $i \leq i_m$. There is only one admissible element on the first row, in position $(1,p-i_m)$, which forces $\sigma(1)=p-i_m$. After erasing the first $p-i_m$ columns there is only one admissible element on the second row, in position $(2,p-i_m+1)$, hence $\sigma(2)=p-i_m+1$. Continuing in this way, by induction one shows easily that $\sigma(i)=(p-i_m-1)+i=\tau^{-1}(i)$ for $i=1,\ldots,i_m$, hence $\sigma=\tau^{-1}$ as claimed.
\end{proof}
This finishes the proof of Proposition \ref{prop:SingleLeibnizTermWithMinimalValuation}, hence it shows that for our special choice of $\beta$ we have $v_K(\det(A))=\mu$. 
It now follows from Equation \eqref{eq:DeterminantsToIndices} and the previous considerations that 
\[
 m(L/K)=N(\pi_K^{\sum_{i=0}^{p-1} \nu_i} ) \cdot N(\pi_K^{\mu})= p^{f_K(\sum_{i=0}^{p-1} \nu_i+\min_{0\leq i\leq p-1}(i e-(p-1)\nu_i))},
\]
which concludes the proof of Theorem \ref{thm:FormulaCyclicExtensionsDegreeP}.
\begin{remark}\label{rmk:MinimalElement}
The proof shows in particular that the element $\beta = \pi_L^a$ (if $\mu=0$) or $\beta=\pi_L^a+\pi_K^{-\nu_{i_m}}f^{i_m}\pi_L^a$ (if $\mu \neq 0$) realises the minimal index.
\end{remark}

\section{Some classical results on cyclic extensions of degree $p$}\label{sect:BertrandiasFerton}
In this section we recover two results, originally due to Bertrandias, Bertrandias and Ferton, concerning the Galois structure of the ring of integers in cyclic extensions of degree $p$: we characterise the associated order and the cases when $\mathcal{O}_L$ is free over $\mathfrak{A}_{L/K}$. The original results may be found in \cite{MR296047}, \cite{MR0374104} and \cite{MR296048}, but these papers do not contain detailed proofs. We remark that in \cite{MR296048} the authors also obtain a characterisation of the almost-maximally ramified extensions $L/K$ for which $\mathcal{O}_L$ is $\mathfrak{A}_{L/K}$-free. Specifically, they show that -- for $L/K$ almost-maximally ramified -- the ring of integers $\mathcal{O}_L$ is free over $\mathfrak{A}_{L/K}$ if and only if in the continued fraction expansion $\frac{t}{p} = a_0 + \frac{1}{a_1+\frac{1}{a_2+\cdots}}=[a_0; a_1,a_2,\ldots,a_N]$ one has $N \leq 4$. While it would probably be possible to also obtain this result by our methods, the proof would involve computations very similar to those of \cite{MR543208}, so we prefer not to treat this case. 
 
Our approach is based on the following corollary, itself a consequence of Proposition \ref{prop:PropertiesOff}.
\begin{corollary}\label{cor:iPlusjGreaterThanP}
Let $0 \leq i,j \leq p-1$ be such that $i+j \geq p$. Then
\[
\pi_K^{-\nu_i} \, f^i \cdot \pi_K^{-\nu_j} f^j \pi_L^a
\]
lies in $\mathcal{O}_L$, and if $L/K$ is not almost-maximally ramified, then it also lies in $\pi_K \mathcal{O}_L$.
\end{corollary}
\begin{proof}
We may write
\[
\pi_K^{-\nu_i} \, f^i \cdot \pi_K^{-\nu_j} f^j \pi_L^a = \pi_K^{-\nu_i-\nu_j} f^{i+j-p} f^p \pi_L^a,
\]
and by Proposition \ref{prop:PropertiesOff} (5) we know $v_L(f^p \pi_L^a) = ep+t+a$. Since every further application of $f$ increases the valuation by at least $t$ (by Proposition \ref{prop:PropertiesOff} (6)), we get
\[
\begin{aligned}
v_L( \pi_K^{-\nu_i} \, f^i \cdot \pi_K^{-\nu_j} f^j \pi_L^a) & \geq v_L(\pi_K^{-\nu_i-\nu_j}) + (i+j-p)t +ep+t+a \\
& = (-\nu_i-\nu_j)p + (i+j-p+1)t + ep + a
\end{aligned}
\]
Using the obvious inequality $\nu_k \leq \frac{a+kt}{p}$ and the fact that $e \geq (p-1)t_0+a=t-t_0$ (Remark \ref{rmk:EGeqA}) we obtain
\[
\begin{aligned}
v_L(\pi_K^{-\nu_i} \, f^i \cdot \pi_K^{-\nu_j} f^j \pi_L^a) & \geq (-a-it-a-jt) + (i+j-p+1)t + p(t-t_0) + a = 0,
\end{aligned}
\]
so $\pi_K^{-\nu_i} \, f^i \cdot \pi_K^{-\nu_j} f^j \pi_L^a$ is integral. If $L/K$ is not almost-maximally ramified, then (again by Remark \ref{rmk:EGeqA}) we have $e \geq t-t_0+1$, which leads to $v_L(\pi_K^{-\nu_i} \, f^i \cdot \pi_K^{-\nu_j} f^j \pi_L^a) \geq p$, that is, this element lies in in $\pi_K \mathcal{O}_L$.
\end{proof}

\subsection{Description of the associated order}
\begin{theorem}\label{thm:StructureOfALK}
Assume $a \neq 0$. Then the associated order $\mathfrak{A}_{L/K}$ is a free $\mathcal{O}_K$-module with basis $\pi_K^{-n_i} f^i$ for $i=0,\ldots,p-1$, where
\[
n_i = \min_{0 \leq j \leq p-1-i} (\nu_{i+j} - \nu_j).
\]
\end{theorem}
\begin{proof}
Let $\lambda$ be an element of $K[G]$. As $1, f, \ldots, f^{p-1}$ is a $K$-basis of $K[G]$, we may write uniquely $\lambda = \sum_{i=0}^{p-1} c_i f^i$ with the $c_i$ in $K$. By Theorem \ref{thm:StructureOfStuff},  $\lambda\in\mathfrak{A}_{L/K}$ if and only if $\lambda(\pi_K^{-\nu_j} f^j \pi_L^a)$ is in $\mathcal{O}_L$ for every $j$.
Taking $j=0$ yields
$
\sum_{i=0}^{p-1} c_i f^i \pi_L^a \in \mathcal{O}_L,
$
and since $v_L(c_i f^i \pi_L^a) = pv_K(c_i) + a+it$ by Proposition \ref{prop:PropertiesOff} (3), we see that the valuations of the terms in the previous sum are pairwise distinct (since they are all different modulo $p$). In particular there can be no cancellation among them, hence $c_i f^i \pi_L^a \in \mathcal{O}_L$ holds for all $i=0,\ldots,p-1$. Using again that the valuation of $f^i\pi_L^a$ is $a+it$ we obtain $v_L(c_i) \geq -a-it$, from which it follows that $v_K(c_i) \geq -\nu_i$ is a necessary condition for $\lambda= \sum_{i=0}^{p-1} c_i f^i\in\mathcal{O}_L$. 
We claim that $\lambda=\sum_{i=0}^{p-1} c_i f^i$ is in $\mathfrak{A}_{L/K}$ if and only if  each summand $c_i f^i$ is. One implication is clear. On the other hand, we have that $\lambda\in \mathfrak{A}_{L/K}$ exactly when, for each $j$,
$$\left(\sum_{i=0}^{p-1} c_i f^{i}\right) \pi_K^{-\nu_j}f^j \pi_L^a\in \mathcal{O}_L.$$ 
We rewrite this last condition as
$$
 \sum_{i=0}^{p-1-j} c_i \pi_K^{-\nu_j} f^{i+j}\pi_L^a +\sum_{i=p-j}^{p-1} c_i \pi_K^{-\nu_j} f^{i+j} \pi_L^a
\in\mathcal{O}_L.
$$
Now, all the summands  $c_i \pi_K^{-\nu_j} f^{i+j} \pi_L^a$ with $i+j\ge p$, namely those of the second sum, are in $\mathcal{O}_L$ by Corollary \ref{cor:iPlusjGreaterThanP},  since $v_K(c_i) \geq -\nu_i$, so the condition reduces to 
\[
 \sum_{i=0}^{p-1-j} c_i \pi_K^{-\nu_j} f^{i+j} \pi_L^a \in \mathcal{O}_L.
\]
By Proposition \ref{prop:PropertiesOff} (3) we have $v_L(c_i \pi_K^{-\nu_j} f^{i+j} \pi_L^a) \equiv a+(i+j)t \equiv (i+j+1) a \pmod{p}$, namely, the valuations of the terms in the above sum are all distinct, and the only possibility for the sum to be integral is that every summand is itself integral. 
This shows that $c_if^i\cdot\pi_K^{-\nu_j}f^j\pi_L^a\in\mathcal{O}_L$,  for each $i$ and $j$, namely that $c_if^i$ must be in $\mathfrak{A}_{L/K}$.

Thus it remains to understand what elements of the form $c_i f^i$ are in $\mathfrak{A}_{L/K}$; equivalently, we require
$
v_L(c_i f^i \cdot \pi_K^{-\nu_j} f^j \pi_L^a) \geq 0$ for all $j=0,\ldots,p-1$. In turn, this is equivalent to
\[
p v_K(c_i) -p\nu_j+ v_L(f^{i+j} \pi_L^a) \geq 0 \; \text{ for all } j=0,\ldots,p-1.
\]
As observed above, for $i+j\ge p$ this is guaranteed by the condition $v_K(c_i) \geq -\nu_i$.
For $j=0,\ldots,p-1-i$, using Proposition \ref{prop:PropertiesOff} (3) we get
\[
 -p\nu_j+ (i+j)t+a \geq -p v_K(c_i) \Leftrightarrow  \nu_{i+j}-\nu_j \geq -v_K(c_i);
\]
since this needs to hold for every $j$, we obtain $v_K(c_i) \geq -n_i$. Putting everything together, we see that $\sum_{i=0}^{p-1} c_i f^i$ is in $\mathfrak{A}_{L/K}$ if and only if $v_K(c_i) \geq -n_i$ for every $i$, that is, $\mathfrak{A}_{L/K}$ is free over $\mathcal{O}_K$ with basis $\{\pi_K^{-n_i} f^i\}_{i=0,\ldots,p-1}$.
\end{proof}

\subsection{On the $\mathfrak{A}_{L/K}$-freeness of $\mathcal{O}_L$}
In this section we show how our approach leads to a new proof of Theorem \ref{thm:IntroFertonBertrandias}. Notice that if $a=0$ then $\mathcal{O}_L$ is free over $\mathfrak{A}_{L/K}$ by Corollary \ref{cor:AEqZeroOLIsAFree}, while if $L/K$ is not almost-maximally ramified, then $a \neq 0$ by Proposition \ref{prop:InequalitiesOneAndt}. Hence in what follows we can prove the theorem under the additional assumption that $a \neq 0$.
We aim to compare the two expressions
\begin{equation}\label{eq:MinimalIndexOverFreeSubmodule}
\sum_{i=0}^{p-1} \nu_i + \min_{0 \leq i \leq p-1} (ei -(p-1)\nu_i)
\end{equation}
and
\begin{equation}\label{eq:IndexOverAssociatedOrder}
\sum_{i=0}^{p-1} n_i,
\end{equation}
which, by Theorems \ref{thm:FormulaCyclicExtensionsDegreeP} and \ref{thm:StructureOfALK}, are the minimal index of $\mathcal{O}_L$ over a free $\mathcal{O}_K[G]$-submodule and the index $[\mathfrak{A}_{L/K} : \mathcal{O}_K[G]]$, respectively. As a consequence of Proposition \ref{productindex}, $\mathcal{O}_L$ is free over $\mathfrak{A}_{L/K}$ if and only if \eqref{eq:MinimalIndexOverFreeSubmodule} and \eqref{eq:IndexOverAssociatedOrder} are equal, so we are reduced to showing the following statement.
\begin{theorem}\label{thm:BertrandiasFertonReformulation}
Let $L/K$ be a degree-$p$ cyclic extension of $p$-adic fields with ramification jump $t$. Let $a$ be the residue class of $t$ modulo $p$. The following hold:
\begin{enumerate}
\item Suppose $a \mid p-1$. Then \eqref{eq:MinimalIndexOverFreeSubmodule} and \eqref{eq:IndexOverAssociatedOrder} are equal, so $\mathcal{O}_L$ is free over $\mathfrak{A}_{L/K}$.
\item Suppose $L/K$ is not almost-maximally ramified and that $\mathcal{O}_L$ is free over $\mathfrak{A}_{L/K}$, so that \eqref{eq:MinimalIndexOverFreeSubmodule} and \eqref{eq:IndexOverAssociatedOrder} are equal. Then $a \mid p-1$.
\end{enumerate}
\end{theorem}

\subsubsection{Sufficiency}
In this section we assume that $a \mid p-1$ and show that \eqref{eq:MinimalIndexOverFreeSubmodule} and \eqref{eq:IndexOverAssociatedOrder} are equal, thus proving (1) in Theorem \ref{thm:BertrandiasFertonReformulation}. 
We start with the following arithmetical lemma:
\begin{lemma}\label{lemma:CFLength2}
Assume that $a \mid p-1$ and set $k=\frac{p-1}{a}$. Then
$
\nu_i = it_0 + \left\lfloor \frac{i}{k} \right\rfloor
$
for $i=0,\ldots,p-1$.
\end{lemma}
\begin{proof} By definition we have $\nu_i = \left\lfloor \frac{i(pt_0+a)+a}{p} \right\rfloor = it_0 + \left\lfloor \frac{(i+1)a}{p} \right\rfloor$, so we have to prove that $\left\lfloor \frac{(i+1)a}{p} \right\rfloor=\left\lfloor \frac{i}{k} \right\rfloor$. We establish this by induction on $i$. For $i \leq k-1$ we have $(i+1)a \leq ka = p-1$, and $\left\lfloor \frac{(i+1)a}{p} \right\rfloor=0$. Now let $i \geq k$: we have
\[
\left\lfloor \frac{(i+1)a}{p} \right\rfloor = \left\lfloor \frac{(i-k+1)a + ka}{p} \right\rfloor = \left\lfloor \frac{(i-k+1)a + (p-1)}{p} \right\rfloor,
\]
which is  easily seen to be equal to $\left\lfloor \frac{(i-k+1)a}{p} \right\rfloor + 1$, provided that $(i-k+1)a$ is not divisible by $p$.

Since $a<p$ and $i-k+1 \leq p-k < p$ we never have $(i-k+1)a \equiv 0 \pmod p$, so we obtain
\[
\left\lfloor \frac{(i+1)a}{p} \right\rfloor = \left\lfloor \frac{(i-k+1)a}{p} \right\rfloor + 1 = \left\lfloor \frac{i-k}{k} \right\rfloor +1 = \left\lfloor \frac{i}{k} \right\rfloor
\]
as claimed.
\end{proof}

\begin{lemma}
Assume that $a \mid p-1$. Then $\min_{0 \leq i \leq p-1} (ei -(p-1)\nu_i)=0$ and $\nu_i =n_i$ for all $i=0,\ldots,p-1$. In particular, \eqref{eq:MinimalIndexOverFreeSubmodule} and \eqref{eq:IndexOverAssociatedOrder} are equal.
\end{lemma}
\begin{proof}
We need to prove that $ei - (p-1) \nu_i \geq 0$ for all $i=0,\ldots,p-1$ (notice that equality holds for $i=0$): this can be seen by writing $k=\frac{p-1}{a}$ and using the inequality $e \geq a +(p-1)t_0$ (Remark \ref{rmk:EGeqA}) and Lemma \ref{lemma:CFLength2}, that yield
\[
ei - (p-1) \nu_i \geq (a +(p-1)t_0)i -  (p-1) \left( \left\lfloor\frac{i}{k}\right\rfloor + it_0 \right) \geq ai - \frac{(p-1)i}{k} = ai - ai=0.
\]
We now turn to the equality $\nu_i=n_i$. Notice that $n_i$ is defined as the minimum of several quantities, one of which is $\nu_i-\nu_0=\nu_i$, so we certainly have $n_i \leq \nu_i$. As for the opposite inequality, we need to prove that $\nu_{i+j} - \nu_j \geq \nu_i$, that is, using Lemma \ref{lemma:CFLength2} again,
\[
(i+j)t_0 + \left\lfloor \frac{i+j}{k} \right\rfloor \geq it_0 + \left\lfloor \frac{i}{k} \right\rfloor + jt_0 + \left\lfloor \frac{j}{k} \right\rfloor,
\]
which is obvious.
\end{proof}

\subsubsection{Necessity}
We now prove that if $\mathcal{O}_L$ is $\mathfrak{A}_{L/K}$-free, and the extension is not almost-maximally ramified, then $a \mid p-1$. %
Let $0 \leq i,j \leq p-1$ be such that $i+j \geq p$. By Corollary \ref{cor:iPlusjGreaterThanP} and the obvious inequality $n_i \leq \nu_i$ we see that $\pi_K^{-n_i} \, f^i \cdot \pi_K^{-\nu_j} f^j \pi_L^a$ is in $\pi_K \mathcal{O}_L$. Now start with $\beta \in \mathcal{O}_L$, which we assume to generate $\mathcal{O}_L$ over $\mathfrak{A}_{L/K}$ and which we represent as a vector $\begin{pmatrix}
d_0 \\ d_1 \\ \vdots \\ d_{p-1}
\end{pmatrix}$ of coordinates in the basis $\pi_K^{-\nu_i} f^i \pi_L^a$ of $\mathcal{O}_L$. We consider the $\mathcal{O}_K$-lattice $\mathfrak{A}_{L/K} \beta$ in $L$, namely, the free $\mathcal{O}_K$-module spanned by
$
\pi_K^{-n_i} \, f^i \beta
$ for $i=0,\ldots,p-1$.
Using the fact that $\pi_K^{-n_i-\nu_j}f^{i+j} \pi_L^a$ lies in $\pi_K \mathcal{O}_L$ for $i+j \geq p$ 
we obtain
\[
\begin{aligned}
\pi_K^{-n_i} f^i \beta = \pi_K^{-n_i} f^i  \sum_{j=0}^{p-1} d_j \pi_K^{-\nu_j} f^j \pi_L^a
\equiv \sum_{k=i}^{p-1} \left(d_{k-i} \pi_K^{-n_i-\nu_{k-i}+\nu_{k}} \right) \pi_K^{-\nu_{k}} f^{k} \pi_L^a \pmod{\pi_K}.
\end{aligned}
\]
The matrix of the lattice $\mathfrak{A}_{L/K} \beta$, expressed with respect to the basis $\{\pi_K^{-\nu_i}f^i\pi_L^a \}$ of $\mathcal{O}_L$, is then congruent modulo $\pi_K$ to
\[
\begin{pmatrix}
d_0		& 0							& 0 & \cdots & 0 \\
d_1		& \pi_K^{-n_1+\nu_{1}} d_0	& 0 & \cdots & 0  \\
\vdots	& \vdots	 & 				  \pi_K^{-n_2+\nu_{2}} d_0 & \cdots & \vdots \\
d_{p-1}	&							& \vdots & & \pi_K^{-n_{p-1}+\nu_{p-1}} d_0
\end{pmatrix}.
\]
The assumption $\mathfrak{A}_{L/K} \beta=\mathcal{O}_L$ implies that the above matrix must be invertible. On the other hand, it is clear that its determinant is congruent modulo $\pi_K$ to $d_0^p \cdot \pi_K^{\sum_{i=1}^{p-1}  (\nu_i-n_i) }$, so it is invertible if and only if this quantity has valuation $0$. Since $v_K(d_0) \geq 0$ and $\nu_i-n_i \geq 0$ for all $i$, this happens if and only if $d_0$ is a $\pi_K$-adic unit and $\nu_i=n_i$ for $i=1,\ldots,p-1$, so that in particular we must have
\[
n_i = \min_{0 \leq j \leq p-1-i} (\nu_{i+j} - \nu_j) = \nu_i,
\]
which implies %
$
\nu_{i+j} - \nu_j \geq \nu_i$ for all $i,j$ such that $i+j \leq p-1$. Set for simplicity $\tilde{\nu}_i := \nu_i - it_0 = \left\lfloor \frac{a(i+1)}{p} \right\rfloor$ and notice that the previous inequality is equivalent to
\begin{equation}\label{eq:nuisuperadditive}
\tilde{\nu}_{i+j} - \tilde{\nu}_j \geq \tilde{\nu}_i \quad \text{for all }i,j\text{ such that }i+j \leq p-1.
\end{equation}
Let $k$ be such that $ka<p\le (k+1)a$; clearly $0\le k\le p-1$ and $a(k+1)=p+r$ with  $r < a$, so that  $\tilde{\nu}_{k-1}=0$ and $\tilde{\nu}_k=1$. 
For $i\ge0$ we have
$$\tilde{\nu}_{i+k}=\left\lfloor \frac{ (i+k+1)a }{p} \right\rfloor =\left\lfloor \frac{ ia+r+p }{p} \right\rfloor\le\left\lfloor \frac{ (i+1)a }{p} \right\rfloor+1=\tilde{\nu}_{i}+\tilde{\nu}_{k},$$
which, together with Equation \eqref{eq:nuisuperadditive}, shows that $\tilde{\nu}_{i+k} = \tilde{\nu}_i+\tilde{\nu}_k$ for each $i \leq p-1-k$. An immediate induction then implies  $\tilde{\nu}_i = \left\lfloor \frac{i}{k} \right\rfloor$ for all $i \leq p-1$.
 Setting $i=p-1$ we then obtain
\[
a = \tilde{\nu}_{p-1}= \left\lfloor \frac{p-1}{k} \right\rfloor \Rightarrow a \leq \frac{p-1}{k},
\]
hence in particular $ak \leq p-1$. We can then set $i=ak \leq p-1$ to obtain
\[
a = \left\lfloor \frac{ak}{k} \right\rfloor= \tilde{\nu}_{ak} = \left\lfloor \frac{(ak+1)a}{p} \right\rfloor \Rightarrow a \leq \frac{(ak+1)a}{p} \Rightarrow p \leq ak+1 \Rightarrow ak \geq p-1,
\]
hence $ak=p-1$ and $a \mid p-1$ as desired. This concludes the proof of Theorem \ref{thm:BertrandiasFertonReformulation} (2).

\renewcommand{\bibfont}{\small}

\bibliographystyle{plainnatabbrv}
\bibliography{biblio}
 

\end{document}